\documentclass[12pt]{article}
\usepackage[utf8]{inputenc}
\usepackage[T1]{fontenc}
\usepackage[margin=30mm]{geometry}
\usepackage{amsfonts}
\usepackage{amsmath}
\usepackage{amssymb}
\usepackage{amsthm}
\usepackage{graphicx}
\usepackage{subfigure} 
\usepackage[dvipsnames]{xcolor}
\usepackage[ruled,vlined]{algorithm2e}
\usepackage{authblk}
\usepackage{hyperref}
\usepackage{tikz}
\usepackage{caption}
\usepackage{subcaption}
\usepackage{comment}

\usetikzlibrary{shapes.geometric, arrows}
\tikzstyle{nodo}= [circle,draw=black,thick, minimum size=20]

\theoremstyle{plain}
\newtheorem{theorem}{Theorem}[section]
\newtheorem{corollary}{Corollary}[section]
\newtheorem{lemma}{Lemma}[section]
\newtheorem{proposition}{Proposition}[section]
\theoremstyle{definition}
\newtheorem{definition}{Definition}[section]
\theoremstyle{remark}
\newtheorem{remark}{Remark}[section]

\providecommand{\keywords}[1]
{\small	
  \textbf{\textit{Keywords---}} #1}

\title{An analysis of Wardrop equilibrium and social optimum in congested transit networks}
\author[b,c]{Victoria M. Orlando}
\author[a]{Iván L. Degano\footnote{Corresponding author. e-mail: ivandegano@mdp.edu.ar}}
\author[b,c,d]{Pablo A. Lotito}
\affil[a]{UNMdP, Facultad de Ciencias Exactas y Naturales, CEMIM}
\affil[b]{Universidad Nacional del Centro de la Provincia de Buenos Aires, Facultad de Ciencias Exactas, PLADEMA, Tandil, Argentina}
\affil[c]{CONICET, Consejo Nacional de Investigaciones Científicas y Técnicas, Buenos
Aires, Argentina}
\affil[d]{Universidad Nacional del Centro de la Provincia de Buenos Aires, Facultad de Ciencias Exactas, NUCOMPA, Tandil, Argentina}
        
\date{}

\begin{document}

\maketitle

\begin{abstract}
The effective design and management of public transport systems are essential to ensure the best service for users. The performance of a transport system will depend heavily on user behaviour. In the common-lines problem approach, users choose which lines to use based on the best strategy for them. While Wardrop equilibrium has been studied for the common-lines problem, no contributions have been made towards achieving the social optimum. In this work, we propose two optimisation problems to obtain this optimum, using strategy flow and line flow formulations. We prove that both optimisation problems are equivalent, and we obtain a characterisation of the social optimum flows. The social optimum makes it possible to compute the price of anarchy (PoA), which quantifies the system's efficiency. The study of the PoA enables the effective design and management of public transport systems, guaranteeing the best service to users.   
\end{abstract} \hspace{10pt}


\section{Introduction}

The study of public transport modelling is crucial for designing and managing efficient, sustainable, and equitable public transport systems. Recent advances in the area have provided valuable tools for modelling public transport systems and facilitating the evaluation of their performance (\cite{Skhosana2021, Nnene2023, Spengler2023}). Many factors influence people's behaviour and their decision to use public transport. One of them is the fare for the service, so it is important to design a fare system that is accessible to all. This will also benefit the system, as attracting more passengers helps to cover costs, thus ensuring the sustainability of the service (\cite{sipus2022defining}). In addition, the desire to use public transport more frequently is affected by satisfaction with the service provided. It is therefore important to know the preferences and attitudes of users when modelling public transport systems, to ensure an attractive and efficient service (\cite{vos2020modeling}). 

Knowing how users are distributed among the different routes and modes of transport is essential for efficient resource allocation, congestion management and service planning. Technological advances allow obtaining part of this information from smart card data (\cite{Agard2006}). Knowing user behaviour helps to decide, for example, where to invest in new public transport lines or stations to meet future demand (\cite{Shakeel2019}). In short, anticipating user behaviour is critical to creating a transportation network that is efficient, resilient, and responsive to the changing needs and preferences of its users. Sometimes, users decide on their travel itinerary based primarily on the time they will need to spend to complete it. This travel time is affected by the in-vehicle travel time and by the waiting time at the origin stop. \cite{Chriqui1975} proposed the \textit{common bus lines} problem, where it is assumed that passengers consider a set of lines to make their trip and will take the one that arrives first at the stop. The common-lines problem has been extensively studied over the years (\cite{spiess1984contributions, SPIESS198983, cominetti2001common}), leading to significant advancements in the formulation and resolution of the flow assignment problem.  

In transportation modelling, two fundamental concepts consider different perspectives on network efficiency and user behaviour: Wardrop Equilibrium (or User Equilibrium) and System Optimum (see \cite{correawardrop} and the references therein). Wardrop Equilibrium assumes that travel times on all used paths between an origin and a destination are equal and less than travel times on unused paths. In other words, users can not improve their own travel time by unilaterally changing routes. On the other hand, System Optimum represents a traffic assignment in which the overall cost of the system is minimised, assuming that users behave cooperatively seeking the best system performance. These principles are crucial for modelling travellers' route choices and understanding the efficiency of transportation networks.

A relevant metric to quantify the system inefficiency that arises when users make selfish and non-cooperative decisions is the price of anarchy, which consists of the ratio between the cost at equilibrium and the optimal system cost. Understanding and mitigating the price of anarchy is crucial for transportation planners and policymakers because it provides information on the discrepancy between individual and collective optimisation. Strategies and interventions can then be devised to align user behaviour with system efficiency, thus reducing congestion and delays and overall improving the sustainability of transportation networks.

The price of anarchy has been studied in different networks with flow circulation, such as the Internet (\cite{Papadimitriou2001}), road networks (\cite{Youn2008}, \cite{Connors}, \cite{Zhang2016}, \cite{Zang2018}) and public services (\cite{KNIGHT2013122}). Our objective is to study the price of anarchy in the common bus lines problem. For this, we need to know the flow assignment according to the Wardrop Equilibrium and according to the social optimum. The former has been studied in the literature (\cite{cominetti2001common, CCF}), but to our knowledge, no work has addressed the social optimum applied to the common-lines problem. In this work, we propose two optimisation problems that allow us to obtain the social optimum in the common-lines problem, and we arrive at a solution characterisation similar to the one existing in the literature for the equilibrium case. Our work focuses on simple networks, considering only two nodes, and is intended to begin extending the study to general networks in the future. Once we have the solution characterisation for the social optimum, we analyse this approach in two small examples. This, together with the study of the equilibrium assignment, allows us to analyse the evolution of the price of anarchy as demand increases. In this way, we can study when the network becomes inefficient, which allows us to evaluate the consequences of selfish user behaviour.

The work is organised as follows: Section \ref{sec:CLP} presents the common-lines problem, the Wardrop Equilibrium and its characterisation and defines the social optimum in this context. Section \ref{sec:SO} is devoted to obtaining the solution characterisation for the social optimum. Once we have both characterisations (equilibrium and social optimum), we can define in Section \ref{sec:poa} the price of anarchy. Next, in Section \ref{sec:numerical_implementation} we present two examples, in which the social cost and the price of anarchy are analysed, considering the same network with different frequency functions. Finally, we present an appendix to develop some demonstrations in detail.

\section{The common-lines problem} 
\label{sec:CLP}

\cite{Chriqui1975} introduced the common-lines approach to study the flow assignment on public transport networks where lines share some route sections. In this problem, users must select a set of lines that they probably use, and choose the one that involves the shortest expected travel time. Over the years, the study of the common-lines problem has been extended considering general networks, introducing the concept of \textit{strategy} adopted by the users, and considering frequency functions that depend on the flow using the line (\cite{spiess1984contributions, SPIESS198983, cominetti2001common}). In this section, we introduce the common-lines problem and two models of passenger assignment. The first one assumes selfish behaviour (Wardrop Equilibrium) and the second one cooperative behaviour (social optimum). We also provide the characterisation of the equilibrium solution obtained by \cite{cominetti2001common}.

We consider the common-lines problem as described in \cite{cominetti2001common}. In the simplest case, it consists of an origin $O$ connected to a destination $D$ by a finite set of bus lines $A=\{a_1, \dots, a_n\}$ (see Figure \ref{fig:simplest_network}). They denote by $v_a$ the flow on each line $a \in A$, consider constant in-vehicle travel time $t_a \geq 0$ and a smooth effective frequency function $f_a:\left[0, \bar{v}_a\right) \rightarrow(0, \infty)$ with $f'_a(v_a)<0$ and $f_a(v_a) \to 0$ when $v_a \to \bar{v}_a$. A frequency function with these characteristics will reflect an increasing waiting time as the line is used by more passengers. The constant $\bar{v}_a >0$ is called the line saturation flow (eventually $\bar{v}_a = \infty$ for some links, including walking links).

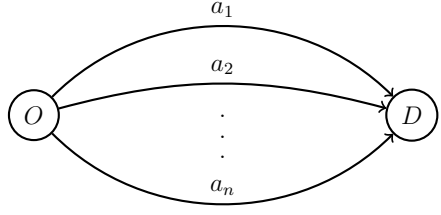
\begin{figure}
    \centering
\begin{tikzpicture}[thick,scale=1, every node/.style={scale=0.8}]
\node at (0,0) [nodo] (O) {$O$};
\node at (5,0) [nodo] (D) {$D$};
\draw [->] (O) to [bend left=45] node[above] {$a_1$} (D);
\draw [->] (O) to [bend left=15] node[above] {$a_2$} (D);
\draw [draw=none] (O) to [bend right=5] node[above] {$.$} (D);
\draw [draw=none] (O) to [bend right=15] node[above] {$.$} (D);
\draw [draw=none] (O) to [bend right=25] node[above] {$.$} (D);
\draw [->] (O) to [bend right=45] node[above] {$a_n$} (D);
\end{tikzpicture}
\caption{The simplest network to analyse the common-lines problem.}
\label{fig:simplest_network}
\end{figure}

To travel from $O$ to $D$, each passenger selects a non empty subset of lines $s \subseteq A$, called the \textit{attractive lines}
or \textit{strategy}, boarding the first incoming bus from this set with available capacity. If $\mathcal{S}$ represents the set of possible strategies, then the number of passengers travelling from $O$ to $D$, denoted by $x \ge 0$, splits among all possible strategies $s \in \mathcal{S}$. That is, $x=\sum_{s \in \mathcal{S}} h_s$, where $h_s$ is the flow of strategy $s \in \mathcal{S}$. Assuming that a passenger using strategy $s$ boards the line $a \in s$ with probability $\pi_a^s=f_a\left(v_a\right) / \sum_{b \in s} f_b\left(v_b\right)$, it turns out that each strategy flow vector $(h_s)_{s\in \mathcal{S}}$ induces a unique vector of line flows $(v_a)_{a\in A}$ through the system of equations
\begin{equation}\label{eq:v}
    v_a = \sum_{s \ni a}h_s \frac{f_a(v_a)}{\sum_{b \in s}f_b(v_b)}, \hspace{3mm} \forall a \in A.
\end{equation}
The expected transit time of each strategy $s$, including the waiting time and in-vehicle travel time, can be represented using this line-flow vector $v$ by the following equation:
\begin{equation}\label{eq:strategycost}
     T_s(v) = \frac{1+ \sum_{a \in s}t_{a}f_{a}(v_{a})}{\sum_{a \in s}f_{a}(v_{a})}.
\end{equation}

To simplify notation, we assume throughout this work a scenario with $n$ bus lines $A=\{1,\dots,n\}$ connecting an origin $O$ to a destination $D$, each one characterised by the in-vehicle travel time $t_i$ and frequency $f_i$. Then, we can rewrite equation \eqref{eq:v} as
\begin{equation}\label{eq:vdelta}
    v_i = \sum_{s \in \mathcal{S}} h_s \delta_{si} \frac{f_i(v_i)}{\sum_{j \in s}f_j(v_j)}, \hspace{3mm} \forall i=1, \dots, n,
\end{equation}
where
$$ \delta_{si}= \begin{cases}
    1 & \text{if } i\in s,\\
    0 & \text{if } i\notin s.
\end{cases}$$

According to Wardrop’s first principle, flows are assigned along strategies with minimal (and therefore equal) cost. Next, we present the formal definition proposed in \cite{cominetti2001common}.

\begin{definition}
It is said that a strategy-flow vector $h \ge 0$ with $\sum_{s \in \mathcal{S}} h_s=x$ is a \emph{user equilibrium} (UE) for the common-lines problem if and only if all strategies carrying flow are of minimal time, that is,
\begin{equation}    
h_{s} >0 \Rightarrow T_{s}(v(h)) =\widehat{T}(v(h)),
\end{equation}
where $\widehat{T}(v(h))=\min_{s\in\mathcal{S}}T_s(v(h))$.
\end{definition} 

Throughout the work, we will refer to the User Equilibrium as the Wardrop Equilibrium. The set of strategy flows satisfying the Wardrop Equilibrium condition is denoted by $H_x$, while $V_x$ represents the set of induced line flows $v(h)$ corresponding to all $h \in H_x$.
The set $V_x$ was characterised in \cite{cominetti2001common} as the optimal solution set of an equivalent optimisation problem, which implies the existence of a constant $\alpha_x \ge 0$ such that $v \in V_x$ if and only if $v \ge 0$ with $\sum_{i=1}^n v_i=x$ and 
\begin{equation} \label{eq:uschar}
v_i = \begin{cases} 0 & \text { if } t_i>\widehat{T}\left(w\left(\alpha_x\right)\right),\\
w_i(\alpha_x) & \text { if } t_i<\widehat{T}\left(w\left(\alpha_x\right)\right), \\ 
0 \leq v_i \le w_i(\alpha_x)  & \text { if } t_i=\widehat{T}\left(w\left(\alpha_x\right)\right), \end{cases}
\end{equation}
considering $w(\alpha) = (w_i(\alpha))_{i=1}^n$, where $w_i(\cdot)$ is the inverse of the differentiable and strictly increasing function $v_i \mapsto v_i / f_i\left(v_i\right)$. Here, the value $v_i/f_i(v_i)$ represents the social waiting cost of line $i$ for the flow $v_i$. In Appendix \ref{app:w} we give some properties of the functions $w_i$ that will be useful in what follows.

Now we define the social optimum as the feasible flow vector that minimises the system's total cost, given by summing over all the strategies the product of the flows $h_s$ times the expected transit times $T_s$.
\begin{definition}
    A strategy-flow vector $h \ge 0$ is a \emph{social optimum} (SO) for the common-lines problem if and only if it is the optimal solution of the following optimisation problem:
     \begin{equation}\label{eq:social_optimimum_original}
        \begin{aligned}
        \min_{h} \quad &\sum_{s\in\mathcal{S}}h_sT_s(v(h))\\
        \textrm{s.t.} \quad& h_s\geq0, &&  s\in \mathcal{S},\\
        &\sum_{s\in \mathcal{S}} h_s=x.
        \end{aligned}
    \end{equation}
\end{definition}

\section{Existence and characterisation of social optimum}\label{sec:SO}

Our goal in this section is to prove the existence of a social optimum $h$ and characterise the corresponding line flows $v(h)$ as the optimal solution of an equivalent optimisation problem. We will provide an analogous characterisation to \eqref{eq:uschar} for the social optimum in the common-lines problem \eqref{eq:social_optimimum_original}.

From \eqref{eq:strategycost}, we can rewrite the above definition as follows:
\begin{equation}
    \begin{aligned}
    \min_{h} \quad &\sum_{s\in\mathcal{S}}\frac{h_s+\sum_{i\in s} h_s t_if_i(v_i(h))}{\sum_{i\in s} f_i(v_i(h))}\\
    \textrm{s.t.} \quad& h_s\geq0, \quad  s\in \mathcal{S},\\
    &\sum_{s\in \mathcal{S}} h_s=x.
    \end{aligned}
\end{equation}
By \eqref{eq:vdelta},
$$ \sum_{s\in\mathcal{S}} \sum_{i\in s} h_s t_i\frac{ f_i(v_i(h))}{\sum_{j\in s} f_j(v_j(h))} = \sum_{i=1}^n \sum_{s \in \mathcal{S}} t_i h_s \delta_{si} \frac{ f_i(v_i(h))}{\sum_{j\in s} f_j(v_j(h))} = \sum_{i=1}^n t_i v_i(h),$$
then we can include $v_i(h)$ in the constraints, thus obtaining
\begin{equation}
    \begin{aligned}
    \min_{v,h} \quad &\sum_{s\in\mathcal{S}}\frac{h_s}{\sum_{i\in s}f_i(v_i)}+\sum_{i=1}^n t_iv_i\\
    \textrm{s.t.} \quad& h_s\geq0, \quad  s\in \mathcal{S},\\
    &\sum_{s\in \mathcal{S}} h_s=x,\\
    &\sum_{s \in \mathcal{S}} h_s \delta_{si} \frac{ f_i(v_i)}{\sum_{j\in s} f_j(v_j)}=v_i.
    \end{aligned}
\end{equation}

Now, we can write this problem as a minimisation problem in $v$ that includes a minimisation in $h$, that is,
\begin{equation}\label{eq:SOv}
    \begin{aligned}
    \min_{v} \quad &\Phi(v)+\sum_{i=1}^n t_iv_i\\
    \textrm{s.t.} \quad& v_i\geq0, &&  i=1, \dots, n,\\
    &\sum_{i=1}^n v_i=x,
    \end{aligned}
\end{equation}
where $\Phi(v)$ is the optimal value of the following optimisation problem parameterised by $v$:
\begin{equation}
    \begin{aligned}
    \min_{h} \quad &\sum_{s\in\mathcal{S}}\frac{h_s}{\sum_{i \in s }f_i(v_i)}\\
    \textrm{s.t.} \quad& h_s\geq0, &&  s\in \mathcal{S},\\
    &\sum_{s\in \mathcal{S}} h_s=x,\\
    &\sum_{s \in \mathcal{S}} h_s \delta_{si} \frac{ f_i(v_i(h))}{\sum_{j\in s} f_j(v_j(h))}=v_i && i=1,\dots,n.
    \end{aligned}
\end{equation}
Because it can be shown that $\sum_{s \in \mathcal{S}} h_s = \sum_{i=1}^n v_i$ and in \eqref{eq:SOv} it is requested that $\sum_{i =1}^n v_i = x$, we can remove the constraint $\sum_{s \in \mathcal{S}} h_s =x$ from $\Phi$. Furthermore, to simplify the expression of $\Phi$, we define $\tau_s = (\sum_{i \in s} f_i(v_i))^{-1}$, and we divide by $f_i(v_i)$ the last constraint, obtaining:
\begin{equation}\label{eq:Phi}
    \begin{aligned}
    \min_{h} \quad &\sum_{s\in\mathcal{S}} h_s\tau_s\\
    \textrm{s.t.} \quad& h_s\geq0, &&  s\in \mathcal{S},\\
    &\sum_{s \in \mathcal{S}} h_s \delta_{si} \tau_s = \frac{v_i}{f_i(v_i)},&&  i=1,\dots,n.
    \end{aligned}
\end{equation}

The existence of the social optimum depends on the following characterisation of the function $\Phi$.
\begin{theorem}
    The optimal value of problem \eqref{eq:Phi} is given by
    $$ \Phi(v) = \max_i \left(\frac{v_i}{f_i(v_i)}\right).$$
\end{theorem}
\begin{proof}
Problem \eqref{eq:Phi} is a linear optimisation problem in $h$ that can be solved by inspection. To simplify the proof, we assume that the values of $\frac{v_i}{f_i(v_i)}$ are decreasing, i.e., 
 $$ \frac{v_i}{f_i(v_i)} \ge \frac{v_{i+1}}{f_{i+1}(v_{i+1})} \quad \forall i=1,\dots,n-1.$$ 
We denote by $\bar{s}$ the strategy that includes all the lines. Since the objective is to minimise the function
$$\sum_{s \in \mathcal{S}} h_s \tau_s = \sum_{\substack{s \in \mathcal{S} \\ s \neq \bar{s} }} h_s \tau_s + h_{\bar{s}}\tau_{\bar{s}}$$
and $\tau_{\bar{s}} \leq \tau_{s}$ for all $s \in \mathcal{S}$, it is convenient to bet everything on $\tau_{\bar{s}}$. This means that it is desirable that as much of the flow as possible uses the strategy $\bar{s}$. That is, we seek $h_{\bar{s}} \geq h_{s}$ for all $s \in \mathcal{S}$. 

From the last constraint of problem \eqref{eq:Phi}, we have that
$$\sum_{s \in \mathcal{S}} \delta_{si} h_s \tau_s = \frac{v_i}{f_i(v_i)} \quad \forall i \in A.$$
Furthermore, because it is a sum of non-negative terms, we have
$$ \delta_{\bar{s}i} h_{\bar{s}} \tau_{\bar{s}} \leq \sum_{s \in \mathcal{S}} \delta_{si} h_s \tau_s,$$
and since $i \in \bar{s}$ for all $i =1,\dots,n$ (so $\delta_{\bar{s}i}=1$ for all $i =1,\dots,n$), we know that 
$$h_{\bar{s}} \tau_{\bar{s}} \leq \frac{v_i}{f_i(v_i)} \quad \forall i =1,\dots,n.$$
From this, we can conclude that the largest value that $h_{\bar{s}}\tau_{\bar{s}}$ can take is the minimum of $\frac{v_i}{f_i(v_i)},$ that is, $\frac {v_n}{f_n(v_n)}$. However, because we have established that $h_{\bar{s}}$ must take the largest possible value and since $\tau_{\bar{s}}$ is fixed, must be $h_{\bar{s}}\tau_{\bar{s}}=\frac {v_n}{f_n(v_n)}$.
Returning to the last constraint of problem \eqref{eq:Phi}, we have that
$$\sum_{s \in \mathcal{S}} \delta_{sn}h_s \tau_s = \frac{v_n}{f_n(v_n)},$$
from where we can obtain
$$\frac{v_n}{f_n(v_n)} = \sum_{ s \neq \bar{s} } \delta_{sn}h_s \tau_s + \delta_{\bar{s}n}h_{\bar{s}} \tau_{\bar{s}} = \sum_{ s \neq \bar{s} } \delta_{sn}h_s \tau_s + \frac {v_n}{f_n(v_n)}$$
and we conclude that
\begin{equation}\label{eq:T2.1_1}
   \sum_{ s \neq \bar{s} } \delta_{sn}h_s \tau_s=0. 
\end{equation}
Furthermore, for any line $i \neq n$, it must be satisfied 
$$\frac{v_i}{f_i(v_i)} = \sum_{s \in \mathcal{S}} \delta_{si}h_s \tau_s = \sum_{ s \neq \bar{s} } \delta_{si}h_s \tau_s + \delta_{\bar{s}i}h_{\bar{s}} \tau_{\bar{s}} = \sum_{ s \neq \bar{s} } \delta_{si}h_s \tau_s + \frac {v_n}{f_n(v_n)}$$
and we have that
\begin{equation}\label{eq:T2.1_2}
    \sum_{ s \neq \bar{s} } \delta_{si}h_s \tau_s = \frac {v_i}{f_i(v_i)} - \frac {v_n}{f_n(v_n)}.
\end{equation}
Up to now, we know that must be $h_{\bar{s}}^{*}\tau_{\bar{s}}= \frac{v_n}{f_n(v_n)}$ and  from \eqref{eq:T2.1_1} we can conclude that $h_s=0$ for each $s \neq \bar{s}$ such that $n \in s$. We can then rewrite the problem \eqref{eq:Phi} as

\begin{equation*}
    \begin{aligned}
    \min_{h} \quad &\sum_{\substack{s \in \mathcal{S} \\ n \notin s}} h_s\tau_s + \frac{v_n}{f_n(v_n)} \\
    \textrm{s.t.} \quad& h_s\geq0, &&  s\in \mathcal{S}, n \notin S, \\
    &\sum_{\substack{s \in \mathcal{S} \\ n \notin s}} \delta_{si}h_s \tau_s = \frac{v_i}{f_i(v_i)}-\frac{v_n}{f_n(v_n)},&&  i \in \{1,\dots, n-1\}.
    \end{aligned}
\end{equation*}
To obtain the flow of the remaining strategies (those that do not contain line $n$), it is sufficient to see that the problem remains the same but with one less line. Now, the line $n$ is not considered, and we have new independent terms given by $\frac{v_i}{f_i(v_i)}-\frac{v_n}{f_n(v_n)}$ maintaining the decreasing property. Repeating the same procedure, we obtain that $h^*_{\tilde{s}}\tau_{\tilde{s}} =  \frac{v_{n-1}}{f_{n-1}(v_{n-1})}-\frac{v_n}{f_n(v_n)}$, where $\tilde{s}$ is the strategy that contains all lines except $n$. The new minimisation problem is as follows:
\begin{equation*}
    \begin{aligned}
    \min_{h} \quad &\sum_{\substack{s \in \mathcal{S} \\ n-1,n \notin s}} h_s\tau_s + \frac{v_{n-1}}{f_{n-1}(v_{n-1})} \\
    \textrm{s.t.} \quad& h_s\geq0, &&  s\in \mathcal{S}, n-1,n \notin s,\\
    &\sum_{\substack{s \in \mathcal{S}\\ n-1,n \notin s}} \delta_{si} h_s \tau_s = \frac{v_i}{f_i(v_i)}-\frac{v_n}{f_n(v_n)}-\frac{v_{n-1}}{f_{n-1}(v_{n-1})},&&  i \in \{1,\dots, n-2\}.
    \end{aligned}
\end{equation*}
Continuing this procedure, the optimal value is $\frac{v_{1}}{f_{1}(v_{1})}$ in the ordered case or $\underset{i}{\max} \frac{v_{i}}{ f_i (v_i) }$ in the general case.

\end{proof}

Calling $\alpha=\underset{i}{\max} \left(\frac{v_i}{f_i(v_i)}\right)>0$ and applying the above Theorem to problem \eqref{eq:SOv}, we have 
\begin{equation}\label{eq:social_optimum_without_w}
    \begin{aligned}
    \min_{v,\alpha} \quad &\sum_{i=1}^n t_iv_i + \alpha\\
    \textrm{s.t.} \quad& v_i\geq0, &&  i =1,\dots,n,\\
    &\sum_{i=1}^n v_i=x,\\
    & \frac{v_i}{f_i(v_i)}\leq\alpha, &&  i =1,\dots,n.
    \end{aligned}
\end{equation}
This formulation is quite similar to that given in Section 1.2 of \cite{cominetti2001common} for the UE problem, considering the function $h$ proposed there as the identity. Since $w_i(\cdot)$ is the inverse function of $v_i \mapsto v_i / f_i\left(v_i\right)$, we finally have a new linear optimisation problem with non-linear constraints that do not include the strategy flow variables $h_s$: 
\begin{equation} \label{eqn:opt}
    \begin{aligned}
    \min_{v,\alpha} \quad &\sum_{i=1}^n t_iv_i + \alpha\\
    \textrm{s.t.} \quad &\sum_{i=1}^n v_i=x,\\
    & 0\leq v_i\leq w_i(\alpha), &&  i =1,\dots,n.
    \end{aligned}
\end{equation}
This problem is convex whenever $-w_i$ is a convex function (in other words, whenever $w_i$ is concave). As we will see below, the chosen frequency function satisfies the condition stated in Proposition \ref{prop:concavity_w}, so that $w_i$ turns out to be a concave function, and the problem \eqref{eqn:opt} is convex.

For this problem, the Karush-Kuhn-Tucker (KKT) conditions implies that there exist multipliers $\lambda, \mu^1_i\geq0, \mu^2_i\geq0$ such that 
\begin{align}
    &\mu_i^1v_i=0 \mbox{ and } \mu_i^2(v_i-w_i(\alpha))=0 \mbox{ for } i =1,\dots,n,\label{eq:kkta} \tag{KKT-a}\\
    &t_i-\lambda-\mu^1_i+\mu^2_i=0, \mbox{ for } i =1,\dots,n, \label{eq:kktb} \tag{KKT-b}\\
    &1-\sum_{i=1}^n \mu_i^2w'_i(\alpha) = 0. \label{eq:kktc} \tag{KKT-c}
\end{align}
Since $\alpha>0$ we have $w_i(\alpha)>0$ for all $i =1,\dots,n$, and then \eqref{eq:kkta} implies $\mu_i^1$ and $\mu^2_i$ cannot both be positive. Furthermore, \eqref{eq:kktb} is equivalent to $ t_i-\lambda=\mu^1_i-\mu^2_i$. Then, combining both facts, yields $\mu^1_i=(t_i-\lambda)_+$ and $\mu^2_i=(\lambda-t_i)_+$. Then equation \eqref{eq:kktc} becomes 
\begin{equation} \label{eq:lambdaalpha}
    \sum_{i=1}^n(\lambda-t_i)_+w'_i(\alpha) = 1
\end{equation}
and we get the following system
\begin{equation} \label{eq:kktsystem}
    \left\lbrace \begin{array}{ll}
        \mu_i^1v_i=0 &  \mbox{ for } i =1,\dots,n, \\
         \mu_i^2(v_i-w_i(\alpha))=0 &  \mbox{ for } i =1,\dots,n, \\
         \mu^1_i=(t_i-\lambda)_+ &  \mbox{ for } i =1,\dots,n, \\
         \mu^2_i=(\lambda-t_i)_+ &  \mbox{ for } i =1,\dots,n, \\
         \sum_{i=1}^n(\lambda-t_i)_+w'_i(\alpha) = 1. &
    \end{array} \right.
\end{equation}

For a fixed $\alpha$ we define the function $\psi_{\alpha}(\lambda)$ as the left side of equation \eqref{eq:lambdaalpha}, i.e., $$\psi_{\alpha}(\lambda)=\sum_{i=1}^n(\lambda-t_i)_+w'_i(\alpha).$$
If we can guarantee the existence and uniqueness of the solution of $\psi_{\alpha}(\lambda)=1$, then there exists a unique $\lambda$ such that the last condition in \eqref{eq:kktsystem} is satisfied. Once $\lambda$ is obtained, we can obtain $\mu_{i}^{1}$ and $\mu_{i}^{2}$ to finally obtain the arc flows $v_i$ for all $i \in A$.

The $\psi_{\alpha}$ graph is shown in Figure \ref{fig:psi_one_alpha}. It is not difficult to see that the function $\psi_{\alpha}$ is continuous, piece-wise linear, strictly increasing, and convex for each $\alpha$ (since $w_i'(\alpha)>0$). Also, $\psi_{\alpha}(\lambda)=0$ if $\lambda$ is less than the minimum of $t_i$' s. From that value, the function strictly increases with a slope that positively jumps every time it passes through a new $t_i$. Therefore, the function will exceed the value of 1, and then, the equation \eqref{eq:lambdaalpha} must have a unique solution. Without loss of generality, we will assume that the values of $t_i$ are increasing, i.e., $t_i < t_{i+1}$ for all $i=1, \dots, n-1$, so $t_1=\min_{i} t_i$. Figure \ref{fig:psi_different_alpha} illustrates the graph of $\psi_{\alpha}$ for different values of $\alpha$. As we can observe, $\psi_{\alpha}$ is a decreasing function of $\alpha$, and the graphs for the specific values $\alpha_2$ and $\alpha_3$ are shown. This and other properties of $\psi_{\alpha}$ are demonstrated in Appendix \ref{sec:app}.

\begin{figure}
    \centering
    \subfigure[$\psi_{\alpha}$ is continuous, piece-wise linear, strictly increasing, and convex for a specific $\alpha$.]
    {
    \begin{tikzpicture}[scale=0.68, transform shape]
    
    \draw[->,black] (0,1)-- node[below=5mm,right=40mm]{$\lambda$}(10,1);
    
    \draw[->] (1,0)--node[above=10mm,left=1mm]{$1$}(1,7);
    \draw [dotted, line width=1pt] (1,4.5)--(10,4.5);
    \draw [dotted]
     (3,1)node[below=4mm,left=-3mm]{$t_1$}--(3,7)
     (5,1)node[below=4mm,left=-3mm]{$t_2$}--(5,2.5)
     (7.5,1)node[below=4mm,left=-3mm]{$t_3$}--(7.5,6)
     (6.42,1)--(6.42,4.5);
    
    \draw
    (0.9,4.5)--(1.1,4.5) 
(3,0.9)--(3,1.1)
    (5,0.9)--(5,1.1)
    (7.5,0.9)--(7.5,1.1)
    (6.42,0.9)--(6.42,1.1);
    \draw[line width=1.5pt]
    (3,1)--(5,2.5)--(7.5,6)--(8,7);
    \draw[fill=black]
    (5,2.5)circle(2pt)
    (7.5,6)circle(2pt)
    (3,1)circle(2pt);
    
    \draw[fill=white]
    (6.42,4.5)circle(3pt);
 
    \draw  (7.3,6.75) node {$\psi_{\alpha}$}
        (6.5,0.5) node {$\lambda(\alpha)$}
        (5,-0.21) node { };

    \end{tikzpicture}
    
    \label{fig:psi_one_alpha}
    }
    \subfigure[$\psi_{\alpha}$ is a decreasing function of $\alpha$.]
    {
    \begin{tikzpicture}[scale=0.68, transform shape]
    
    \draw[->,black] (0,1)-- node[below=5mm,right=40mm]{$\lambda$}(10,1);
    
    \draw[->] (1,0)--node[above=10mm,left=1mm]{$1$}(1,7);
    \draw [dotted, line width=1pt] (1,4.5)--(10,4.5);
    \draw [dotted]
     (3,1)node[below=4mm,left=-3mm]{$t_1$}--(3,7)
     (5,1)node[below=4mm,left=-3mm]{$t_2$}--(5,7) 
     (7.5,1)node[below=4mm,left=-3mm]{$t_3$}--(7.5,7)
     (6.42,1)--(6.42,4.5);
    
    \draw
    (0.9,4.5)--(1.1,4.5) 
    (3,0.9)--(3,1.1)
    (5,0.9)--(5,1.1)
    (7.5,0.9)--(7.5,1.1)
    (6.42,0.9)--(6.42,1.1);
    \draw[line width=1.5pt]
    (3,1)--(5,4.5)--(5.5,7)
    (3,1)--(5,2.5)--(7.5,6)--(8,7)
    (3,1)--(5,1.75)--(7.5,4.5)--(8.9,7);
    \draw[fill=black]
    (5,2.5)circle(2pt)
    (7.5,6)circle(2pt)
    (3,1)circle(2pt)
    (5,1.75)circle(2pt);
    
 \draw[fill=white]
    (5,4.5)circle(3pt)
    (6.42,4.5)circle(3pt)
    (7.5,4.5)circle(3pt);
 
 \draw  (5.9,6.75) node {$\psi_{\alpha_2}$}
        (8.2,6.75) node {$\psi_{\alpha}$}
        (9.2,6.75) node {$\psi_{\alpha_3}$}
        (5,0) node {$\lambda(\alpha_2)$}
        (6.5,0.5) node {$\lambda(\alpha)$}
        (7.5,0) node {$\lambda(\alpha_3)$};
    \end{tikzpicture}
    \label{fig:psi_different_alpha}
    }
    \caption{Graph of function $\psi_{\alpha}$ for a specific $\alpha$ and for different increasing values of $\alpha$.}
    \label{fig:psi_complete}
\end{figure}
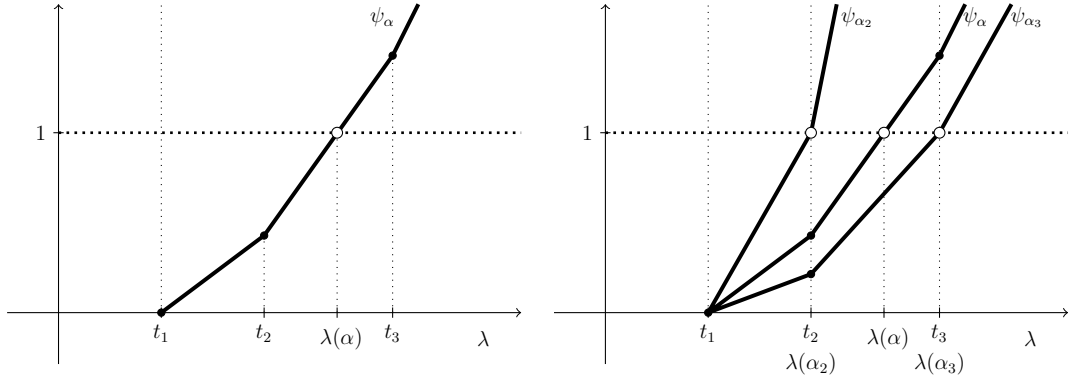

This allows defining an application that maps each $\alpha$ with the solution of equation $\psi_{\alpha}(\lambda)=1$. We denote this function by $\bar{\lambda}$. To proceed with the social optimum characterisation, we establish some preliminary results for $\bar{\lambda}$.

\begin{proposition}
Assume that the functions $w_i\in C^2(0,\infty)$ are strictly concave for all $i=1,\dots,n$. Thus, the function $\bar{\lambda}$ is an increasing function.
\end{proposition}
\begin{proof}
    Let $\hat{\alpha}, \tilde{\alpha} \in (0, \infty)$ be such that $\hat{\alpha}< \tilde{\alpha}$, and let $\hat{\lambda}=\bar{\lambda}(\hat{\alpha})$ and $\tilde{\lambda}=\bar{\lambda}(\tilde{\alpha})$. Then $\psi_{\hat{\alpha}}(\hat{\lambda})= \psi_{\tilde{\alpha}}(\tilde{\lambda})=1$. By Proposition \ref{prop:psidecr},
    $$ \psi_{\hat{\alpha}}(\tilde{\lambda}) > \psi_{\tilde{\alpha}}(\tilde{\lambda}) = \psi_{\hat{\alpha}}(\hat{\lambda}).$$
    Since $\psi_{\alpha}$ is an increasing function for a fixed value of $\alpha$, from above it follows that $\hat{\lambda}<\tilde{\lambda}$, or equivalently, $\bar{\lambda}(\hat{\alpha})<\bar{\lambda}(\tilde{\alpha})$. Then we get the result. 
\end{proof}

We can obtain explicit formulas for the $\bar{\lambda}$ function. In the case where there is no $\alpha_k$ such that $\psi_{\alpha_k}(t_k)=1$ for all $k=2,\dots,n$, we get $\bar{\lambda}(\alpha)>t_n$ for all $\alpha \in (0, \infty)$ by Proposition \ref{prop:chain_alpha}. Then, the equation that defines $\bar{\lambda}(\alpha)$ is $$\sum_{i=1}^{n}(\bar{\lambda}(\alpha)-t_i)w'_i(\alpha) = 1,$$ 
which is easy to solve for 
\begin{equation} \label{eq:lambdanew}
    \bar{\lambda}(\alpha)= \frac{1+\sum_{i=1}^{n}t_iw'_i(\alpha) }{\sum_{i=1}^{n}w'_i(\alpha)}.
\end{equation}

On the other hand, assume there exists $\alpha_j \in (0, \infty)$ such that $\psi_{\alpha_j}(t_j)=1$, or equivalently, $\bar{\lambda}(\alpha_j)=t_j$. Without loss of generality, we assume that $j$ is the first index with this property. Then, by Proposition \ref{prop:chain_alpha}, for each $t_k$ with $k=j,\dots,n$, there exist $\alpha_k \in (0,\infty)$ such that $\bar{\lambda}(\alpha_k)=t_k$. 

Now, the function $\bar{\lambda}$ can be defined in the intervals $(0,\alpha_j),\ (\alpha_{j},\alpha_{j+1}), \dots,(\alpha_n,\infty),$ by inverting the last equation in \eqref{eq:kktsystem}. More precisely, for $j<k\le n-1$ and $\alpha \in (\alpha_{k},\alpha_{k+1})$ we have that $\bar{\lambda}(\alpha)\in (t_{k},t_{k+1})$ so the equation that defines $\bar{\lambda}(\alpha)$ is $$\sum_{i=1}^{k}(\bar{\lambda}(\alpha)-t_i)w'_i(\alpha) = 1,$$ from which it is easy to solve for 
$$
    \bar{\lambda}(\alpha)= \frac{1+\sum_{i=1}^{k}t_iw'_i(\alpha) }{\sum_{i=1}^{k}w'_i(\alpha)}.
$$
In the same way, for $\alpha > \alpha_n$, we have that $\bar{\lambda}(\alpha) > t_n$, so similarly as before, we can obtain 
$$
    \bar{\lambda}(\alpha)= \frac{1+\sum_{i=1}^{n}t_iw'_i(\alpha) }{\sum_{i=1}^{n}w'_i(\alpha)}.
$$
Finally, since $j$ is the first index such that equation \eqref{eq:lambdaalpha} has solution for $t_j$, there does not exist $\alpha_{j-1} \in (0, \infty)$ such that $\psi_{\alpha_{j-1}}(t_{j-1})=1$. Then, from Proposition \ref{prop:chain_alpha}, we get that there is no $\alpha \in (0, \infty)$ such that $\psi_{\alpha}(\lambda)=1$ for $\lambda<t_{j-1}$. Therefore, for $0< \alpha < \alpha_j$, we have that $t_{j-1}< \bar{\lambda}(\alpha) < t_j$, and 
$$
   \bar{\lambda}(\alpha)= \frac{1+\sum_{i=1}^{j}t_iw'_i(\alpha) }{\sum_{i=1}^{j}w'_i(\alpha)}.
$$

Then, under the hypothesis of Corollary \ref{cor:alpha_k} and some abuse of notation,
\begin{equation} \label{eq:lambdafunct}
    \bar{\lambda}(\alpha)= \left\lbrace \begin{array}{cl}
        \frac{1+\sum_{i=1}^{k}t_iw'_i(\alpha) }{\sum_{i=1}^{k}w'_i(\alpha)} & \mbox{if } \alpha \in (\alpha_k, \alpha_{k+1}) \mbox{ for some } k=j-1,\dots,n, \\
        \\
        t_k & \mbox{if } \alpha=\alpha_k \mbox{ for some } k=j,\dots,n,
    \end{array} \right.
\end{equation}
where $\alpha_{j-1}=0$ and $\alpha_{n+1}=+\infty$. 

\begin{theorem}
    If the functions $w_i$ are strictly concave and $w_i'$ is continuous for all $i=1, \dots, n$, then the function $\bar{\lambda}$ is continuous.
\end{theorem}
\begin{proof}
    For the case where there is no $\alpha_k$ such that $\psi_{\alpha_k}(t_k)=1$ for all $k=2,\dots,n$, $\bar{\lambda}$ is defined by \eqref{eq:lambdanew}, which will be continuous as $w_i'$ is continuous for each $i=1, \dots,n$ (since we know that $w_i’(\alpha)>0$ for $\alpha>0$). Assume now there exists $\alpha_j \in (0, \infty)$ such that $\psi_{\alpha_j}(t_j)=1$. Let $k=j-1 \dots, n$ and consider $\alpha \in (\alpha_k, \alpha_{k+1})$. As we saw previously, the function $\bar{\lambda}$ will be \eqref{eq:lambdafunct},
    \begin{equation*}
        \bar{\lambda}(\alpha) = \frac{1+ \sum_{i=1}^{k}t_i w_i'(\alpha)}{\sum_{i=1}^{k}w_i'(\alpha)}
    \end{equation*}
    which will be continuous for the same reason mentioned above. Therefore, all that remains is to prove the continuity in $\alpha_k$ for $k=j,\dots,n$. We know that $\bar{\lambda}(\alpha_k)=t_k$; in that case, 
    $$\psi_{\alpha_k}(t_k)=\sum_{i=1}^{k-1} (t_k-t_i)w_i'(\alpha_k)=1. $$
    Then
    \begin{equation} \label{eq:tk}
        t_k=\frac{1+\sum_{i=1}^{k-1} t_iw_i'(\bar{\alpha_k})}{\sum_{i=1}^{k-1} w_i'(\bar{\alpha_k})}.
    \end{equation} 
    Now, analysing the left limit, we have:
    \begin{equation*}
        \begin{split}
            \lim\limits_{\alpha \to \alpha_k^{-}} \bar{\lambda}(\alpha) &= \lim\limits_{\alpha \to \alpha_k^{-}} \frac{1+ \sum_{i=1}^{k-1}t_iw_i'(\alpha)}{\sum_{i=1}^{k-1}w_i'(\alpha)} \\
            &= \frac{1+ \sum_{i=1}^{k-1}t_iw_i'(\alpha_k)}{\sum_{i=1}^{k-1}w_i'(\alpha_k)} \\
            &= t_k,
        \end{split}
    \end{equation*}
    where the last equality follows from \eqref{eq:tk}. Finally, the right limit is: 
    \begin{equation*}
        \begin{split}
            \lim\limits_{\alpha \to \alpha_k^{+}} \bar{\lambda}(\alpha) &= \lim\limits_{\alpha \to \alpha_k^{+}} \frac{1+ \sum_{i=1}^{k}t_iw_i'(\alpha)}{\sum_{i=1}^{k}w_i'(\alpha)} \\
            &= \frac{1+ \sum_{i=1}^{k}t_iw_i'(\alpha_k)}{\sum_{i=1}^{k}w_i'(\alpha_k)} \\
            &= t_k,
        \end{split}
    \end{equation*}
     where the last equality holds because 
    \begin{equation*}
        \begin{split}
            1=\psi_{\alpha_k}(t_k)=\sum_{i=1}^{k-1}(t_k-t_i)w_i'(\alpha_k) &= \sum_{i=1}^{k-1}(t_k-t_i)w_i'(\alpha_k) + (t_k-t_k)w_i'(\alpha_k) \\
            &= \sum_{i=1}^{k}(t_k-t_i)w_i'(\alpha_k) \\
            &= t_k \sum_{i=1}^{k}w_i'(\alpha_k) - \sum_{i=1}^{k}t_iw_i'(\alpha_k)
        \end{split}
    \end{equation*}
    so it is evident that $t_k = \frac{1+ \sum_{i=1}^{k}t_iw_i'(\alpha_k)}{\sum_{i=1}^{k}w_i'(\alpha_k)}$. Therefore, $\bar{\lambda}$ is continuous in $\alpha_k$ for $k=j, \dots , n$, and then continuous in its domain.

\end{proof}

Now, let $\alpha>0$, there exist $\bar{\lambda}(\alpha)$ such that \eqref{eq:lambdaalpha} holds. Then, the relations given in the system \eqref{eq:kktsystem} implies that $v_i=0$ if $\bar{\lambda}(\alpha)<t_i$ and $v_i=w_i(\alpha)$ if $\bar{\lambda}(\alpha)>t_i$, while for the remaining lines we have $0\le v_i \le w_i(\alpha)$. Hence
$$\sum_i\{w_i(\alpha):t_i<\bar{\lambda}(\alpha) \} \le \sum_{i=1}^n v_i \le \sum_i\{w_i(\alpha):t_i \le \bar{\lambda}(\alpha) \}$$
If we denote by
\begin{equation}\label{eq:x_hat}
    \hat{x}(\alpha) = \sum_i\{w_i(\alpha):t_i<\bar{\lambda}(\alpha) \},
\end{equation}
\begin{equation}\label{eq:x_check}
    \check{x}(\alpha) = \sum_i\{w_i(\alpha):t_i\leq\bar{\lambda}(\alpha) \}, 
\end{equation}
then $x \in [\hat{x}(\alpha),\check{x}(\alpha)]$. The functions $\hat{x}(\alpha)$ and $\check{x}(\alpha)$ are increasing, continuous, and equal, except at the values $\alpha_k$ such that $\bar{\lambda}(\alpha_k)=t_k$ where we have $\check{x}(\alpha_k)>\hat{x}(\alpha_k)$ (Figure \ref{fig:x}). Moreover, $\hat{x}(0)=\check{x}(0)=0$ and $\hat{x}(\alpha)=\check{x}(\alpha) \rightarrow \sum_{i=1}^n \bar{v}_i$ when $\alpha \rightarrow \infty$. Therefore, for each $x \in (0, \sum_{i=1}^n \bar{v}_i)$, there exists a unique $\alpha_x$ such that $x\in[\hat{x}(\alpha_x),\check{x}(\alpha_x)]$.

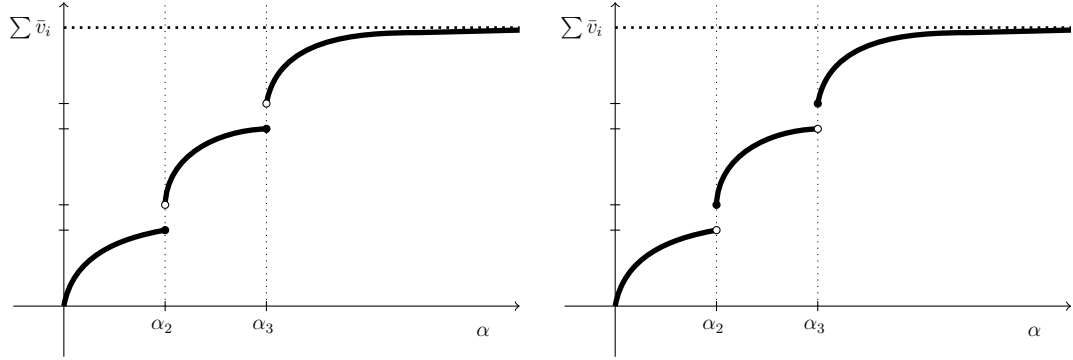
\begin{figure}[t]
    \centering
    \subfigure[Graph of the function $\hat{x}$ defined in \eqref{eq:x_hat}.]
    {
    \begin{tikzpicture}[scale=.67, transform shape]
    
    \draw[->,black] (0,1)-- node[below=5mm,right=40mm]{$\alpha$}(10,1);
    
    \draw[->] (1,0)--node[above=30mm,left=1mm]{$\sum\bar{v}_i$}(1,7);
    \draw [dotted, line width=1pt] (1,6.5)--(10,6.5);

    \draw [black, line width=2pt]
    (1,1) to[out=80,in=-170] (3,2.5) 
    (3,3) to[out=88,in=-178] (5,4.5)
    (5,5) to[out=80,in=-180] (8,6.4)
    (8,6.4) to (10,6.45)
    ;
    
    \draw [dotted]
     (3,1)node[below=4mm,left=-3mm]{$\alpha_2$}--(3,7)
     (5,1)node[below=4mm,left=-3mm]{$\alpha_3$}--(5,7)
     ;
    
    \draw
    (3,0.9)--(3,1.1)
    (5,0.9)--(5,1.1)
    (.9,2.5)    -- (1.1,2.5)
    (.9,3)    -- (1.1,3)
    (.9,4.5)    -- (1.1,4.5)
    (.9,5)    -- (1.1,5)
    ;

    \draw[black, fill=black]
    (3,2.5)circle(2pt)
    (5,4.5)circle(2pt);

    \draw[black, fill=white]
    (3,3)circle(2pt)
    (5,5)circle(2pt);
    
    \end{tikzpicture}
    \label{fig:hatx}
 
    }
    \subfigure[Graph of the function $\check{x}$ defined in \eqref{eq:x_check}.]
    {
    \begin{tikzpicture}[scale=.67, transform shape]
    
    \draw[->,black] (0,1)-- node[below=5mm,right=40mm]{$\alpha$}(10,1);
    
    \draw[->] (1,0)--node[above=30mm,left=1mm]{$\sum\bar{v}_i$}(1,7);
    \draw [dotted, line width=1pt] (1,6.5)--(10,6.5);

    \draw [black, line width=2pt]
    (1,1) to[out=80,in=-170] (3,2.5) 
    (3,3) to[out=88,in=-178] (5,4.5)
    (5,5) to[out=80,in=-180] (8,6.4)
    (8,6.4) to (10,6.45)
    ;
    
    \draw [dotted]
     (3,1)node[below=4mm,left=-3mm]{$\alpha_2$}--(3,7)
     (5,1)node[below=4mm,left=-3mm]{$\alpha_3$}--(5,7)
     ;
    
    \draw
    (3,0.9)--(3,1.1)
    (5,0.9)--(5,1.1)
    (.9,2.5)    -- (1.1,2.5)
    (.9,3)    -- (1.1,3)
    (.9,4.5)    -- (1.1,4.5)
    (.9,5)    -- (1.1,5)
    ;

    \draw[black, fill=white]
    (3,2.5)circle(2pt)
    (5,4.5)circle(2pt);

    \draw[black, fill=black]
    (3,3)circle(2pt)
    (5,5)circle(2pt);
    
    \end{tikzpicture}
    \label{fig:xcheck}

    }
    \caption{The functions $\hat{x}(\alpha)$ and $\check{x}(\alpha)$. Note that they are equal, except at the values $\alpha_k$ such that $\bar{\lambda}(\alpha_k)=t_k$ where we have $\check{x}(\alpha_k)>\hat{x}(\alpha_k)$.}
    \label{fig:x}
\end{figure}

We are ready to provide a direct characterisation of the optimum line-ﬂows $v$. The next Theorem will be useful for stating a simple model for general transit networks by working directly in terms of arc-ﬂows and avoiding explicitly dealing with the notion of strategy.

\begin{theorem} \label{teo:characterization}
     Let $x\in(0,\sum_{i=1}^n \bar{v}_i)$, and assume that the functions $w_i \in C^2(0,\infty)$ are strictly concave for all $i=1, \dots, n$. Set $\bar{\lambda}$ as the function defined by \eqref{eq:lambdafunct}, with $\alpha_x$ the unique solution of $x\in[\hat{x}(\alpha),\check{x}(\alpha)]$. The social optimum is given by 
     \begin{equation}\label{eq:characterization_so}
        v_i =
        \begin{cases} 
             0& \text{ if } t_i > \bar{\lambda}(\alpha_x),  \\
             w_i(\alpha_x)& \text{ if } t_i < \bar{\lambda}(\alpha_x),  \\
             0\leq v_i\leq w_i(\alpha_x)& \text{ if } t_i = \bar{\lambda}(\alpha_x).
        \end{cases}        
    \end{equation} 
\end{theorem}

The last Theorem implies that for each $x \in (0, \sum_{i=1}^n \bar{v}_i)$, there exists at least one social optimum. This optimum will be unique unless there are two or more lines with $t_i=\bar{\lambda}(\alpha)$.

Note that in the case where there is no $\alpha_k \in (0, \infty)$ such that $\bar{\lambda}(\alpha_k)=t_k$ for all $k=2,\dots,n$, the functions $\hat{x}(\alpha)$ and $\check{x}(\alpha)$ are equal in all their domains. Furthermore, by Proposition \ref{prop:chain_alpha}, $\bar{\lambda}(\alpha_x)> t_i$ for all $i$, and then $v_i=w_i(\alpha_x)$. On the other hand, if there exists $\alpha_k \in (0, \infty)$ such that $\bar{\lambda}(\alpha_k)=t_k$ for some $k$ and $x$ does not belong to any of the intervals $[\hat{x}(\alpha_k),\check{x}(\alpha_k)]$, we have $v_i=w_i(\alpha_x)$ for all the lines with $t_i<\bar{\lambda}(\alpha_x)$ and $v_i=0$ for the lines with $t_i>\bar{\lambda}(\alpha_x)$. When $x \in [\hat{x}(\alpha_k),\check{x}(\alpha_k)]$ we have $\alpha_x=\alpha_k$ and $\bar{\lambda}(\alpha_x)=t_k$. Note that in this scenario, a ﬂow increment (or reduction) in the range $[\hat{x}(\alpha_k),\check{x}(\alpha_k)]$ only modifies the optimum flow on line $k$. In this situation, every optimum flow $v_k$ may have $0 < v_k /f_k (v_k) < \alpha_x$ for some line $k$.

Theorem \ref{teo:characterization} is weak in that, in most cases, it is not possible to recover the strategy flows from the line flows. However, from the point of view of public transport decision-makers, it is sufficient to know the number of passengers that must use each line to guarantee optimal operation of the system.

\section{Price of anarchy}\label{sec:poa}

In the previous sections, we presented the common-lines problem for the case of an origin and a destination and the characterisation of the equilibrium solution and the social optimum. This allows us to analyse the public transport network's performance when users behave selfishly (seeking to minimise their travel time) compared to cooperative behaviour (seeking the system's optimal performance). For this, we recall what we mean by \textit{social cost} in the common-lines problem. Given a strategy flow assignment, the social cost of a transport network is given by: 
\begin{equation}\label{eq:social_cost}
    SC = \sum_{s \in \mathcal{S}}h_sT_s(v(h)).
\end{equation}
Equivalently, as a consequence of the results obtained in Section \ref{sec:SO}, given an arc flow assignment the social cost of the transport network can be obtained as
\begin{equation}\label{eq:social_cost_arcs}
    SC = \sum_{i=1}^n t_iv_i + \max_i \left( \frac{v_i}{f_i(v_i)}\right).
\end{equation}

If the strategy-flow vector $h$ (or the line-flow vector $v$) is a social optimum, the social cost is the minimum possible, and we denote this optimal social cost as $OSC$. If, instead, the strategy-flow vector $h^*$ is an equilibrium, each strategy $s$ with $h_s^*>0$ satisfy that $T_s(v(h^*))=\widehat{T}(v(h^*))$, where $\widehat{T}(v(h^*))=\min_{s\in\mathcal{S}}T_s(v(h^*))$. Therefore, the social cost associated with an equilibrium can be obtained as follows:

\begin{equation}\label{eq:social_cost_equilibrium}
    WSC  = \sum_{s \in \mathcal{S}}h_s^*T_s(v(h^*)) 
    = \sum_{s \in \mathcal{S}}h_s^* \widehat{T}(v(h^*)) 
    = \widehat{T}(v(h^*)) \sum_{s \in \mathcal{S}}h_s^* 
    = \widehat{T}(v(h^*)) x.
\end{equation}

When users are permitted to behave selfishly, the social cost is greater than or equal to the optimal system cost ($WSC \geq OSC$). In \cite{Rough2002}, the concept of \textit{price of anarchy} is used to quantify the system’s inefficiency when users follow individual and non-cooperative objectives. This concept has been used in different routing games, for example, public services (\cite{KNIGHT2013122}) or traffic networks (\cite{Youn2008}, \cite{Zang2018}), but to the best of our knowledge, its application in the common-lines problem has not been studied. As introduced in \cite{KNIGHT2013122}, the price of anarchy can be used as an indicator of the inefficiency of a given public transport system, comparing the network performance when users behave selfishly with the optimal state of the system. Specifically, the price of anarchy is given by the ratio between the social cost associated with the Wardrop equilibrium ($WSC$) and the optimal system cost ($OSC$):

\begin{equation}\label{eq:price_of_anarchy}
    PoA = \frac{WSC}{OSC}.
\end{equation}
Based on this formulation, if the price of anarchy is greater than $1$, the system becomes inefficient in the equilibrium situation (when users behave selfishly). With our characterisation of the social optimum, we now have the necessary tools to calculate the price of anarchy for the common-lines problem. This allows us to analyse different scenarios, detect under which conditions the system becomes inefficient, and thus make decisions that allow redirection of flow over the different lines in such a way that the best service is guaranteed to the users.

\section{Numerical Implementation}
\label{sec:numerical_implementation}

In this section, we propose two examples to analyse the equilibrium assignment and the social optimum. We will study the evolution of the social cost in the equilibrium situation and the optimal system cost as demand increases, which will allow us to analyse the price of anarchy and, consequently, the efficiency of the network. 

In both examples, we use the same travel time, bus capacity, and nominal frequency data. However, we use different effective frequency functions extracted from the literature. The first example is approached from the characterisation of the equilibrium solution and the social optimum, given in Equation \eqref{eq:uschar} and Theorem \ref{teo:characterization}, respectively. The second example is approached by solving the corresponding optimisation problems. We will consider a network with two nodes and two arcs connecting them, as shown in Figure \ref{fig:example_two_nodes}. Travel times, bus capacity, and nominal frequencies used will be detailed in the examples.

\begin{figure}
    \centering
\begin{tikzpicture}[thick,scale=0.6, every node/.style={scale=0.8}]
\node at (0,0) [nodo] (O) {$o$};
\node at (5,0) [nodo] (D) {$d$};
\draw [->] (O) to [bend left=45] node[above] {$1$} (D);
\draw [->] (O) to [bend right=45] node[below] {$2$} (D);
\end{tikzpicture}
\caption{Network with one origin ($o$), one destination ($d$) and two arcs ($A=\{1,2\}$).}
\label{fig:example_two_nodes}
\end{figure}
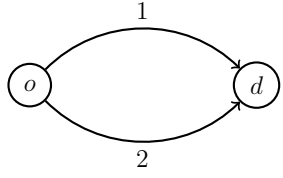

\subsection{Example from solution characterisation}
\label{sec:example_charac}

First, we will present the example given in \cite{cominetti2001common} for the common-lines problem under congestion and calculate the social optimal line flow. Assume a bus stop served by two lines with travel times $t_1<t_2$, independent Poisson arrivals of rate $\mu_1$ and $\mu_2$, and equal capacity $K$. In that case, the analytic expressions for each effective frequency function are
\begin{equation} \label{eq:freqtoy} 
f_i(v_i)=v_i\left( \frac{1}{\rho_i(v_i)} -1 \right) \quad i=1,2,
\end{equation}
with $\bar{v}_i=K\mu_i$ and $\rho_i(v_i)$ the unique solution $\rho \in [0,1)$ of the equation $\mu_i(\rho+\rho^2+\dots+\rho^K)=v_i$. This function is differentiable with $f'_i(v_i)<0$ and $f_i(v_i) \rightarrow 0^{+}$ when $v_i \rightarrow \bar{v}_i$. The inverse $w_i(\alpha)$ of $v_i \to \frac{v_i}{f_i(v_i)}$ is
$$
w_i(\alpha)=\mu_i \alpha\left(1-[\alpha /(1+\alpha)]^K\right) \quad i=1,2,
$$
which is a concave increasing function. Furthermore 
$$ w'_i(\alpha)=\mu_i \left(1-[\alpha /(1+\alpha)]^K [1+K/(1+\alpha)]\right) \quad i=1,2,$$
which is a differentiable function. Then, the assumptions made in the preceding sections hold for the effective frequency function \eqref{eq:freqtoy}.

Therefore, by Theorem \eqref{teo:characterization}, for every $x \in (0,\bar{v}_1+\bar{v}_2)$ there exists a unique social optimum. To compute it explicitly, set $\mu=\frac{\mu_1}{\mu_1+\mu_2}$ and let $\alpha_2$ be the solution of $\bar{\lambda}(\alpha)=t_2$. That is
\begin{equation} \label{eqn:conditionexample}
    \psi_{\alpha}(t_2)=1 \Leftrightarrow (t_2-t_1)w'_1(\alpha)=1  \Leftrightarrow w'_1(\alpha)=\frac{1}{t_2-t_1}.
\end{equation}

By Lemma \ref{lem:alphak}, we know that if $$f_1(0)(t_2-t_1)=w_1'(0)(t_2-t_1)=\mu_1(t_2-t_1)<1$$
there is no $\alpha_2$ solving $\bar{\lambda}(\alpha)=t_2$, and then by Proposition \ref{prop:chain_alpha}, $\bar{\lambda}(\alpha)>t_2$ for all $\alpha \in (0,\infty)$. Therefore, according to Theorem \ref{teo:characterization}, there exist $\alpha_x$ such that $v_1^{\mbox{\tiny so}}=w_1(\alpha_x)$ and $v_2^{\mbox{\tiny so}}=w_2(\alpha_x)$. From the constrain $v_1^{\mbox{\tiny so}}+v_2^{\mbox{\tiny so}}=x$ and the definition of $w_i$, it follows that $v_1^{\mbox{\tiny so}}=\mu x$ and $v_2^{\mbox{\tiny so}}=(1-\mu)x$.

Consider now the case $\mu_1(t_2-t_1)\ge 1$, by Lemma \ref{lem:alphak} there exist $\alpha^{\mbox{\tiny so}}_2$ such that $\bar{\lambda}(\alpha^{\mbox{\tiny so}}_2)=t_2$. Let $l^{\mbox{\tiny so}}=\hat{x}(\alpha^{\mbox{\tiny so}}_2)$ and $u^{\mbox{\tiny so}}=\check{x}(\alpha^{\mbox{\tiny so}}_2)$: 
\begin{itemize}
    \item If $x<l^{\mbox{\tiny so}}$, then $\alpha_x <\alpha^{\mbox{\tiny so}}_2$ and $\bar{\lambda}(\alpha_x)<t_2$. Then from Theorem \ref{teo:characterization} we get $v^{\mbox{\tiny so}}_2=0$ and so $v^{\mbox{\tiny so}}_1=x$.
    \item If $x=l^{\mbox{\tiny so}}$, then $x=w_1(\alpha_x)$ with $\alpha_x =\alpha^{\mbox{\tiny so}}_2$ and $t_1<\bar{\lambda}(\alpha_x)=t_2$. Then from Theorem \ref{teo:characterization} we get $v^{\mbox{\tiny so}}_1=w_1(\alpha^{\mbox{\tiny so}}_2)=x$ and so $v^{\mbox{\tiny so}}_2=0$. 
    \item If $l^{\mbox{\tiny so}}<x<u^{\mbox{\tiny so}}$, then $\alpha_x =\alpha^{\mbox{\tiny so}}_2$ and $t_1<\bar{\lambda}(\alpha_x)=t_2$. Then from Theorem \ref{teo:characterization} we get $v^{\mbox{\tiny so}}_1=w_1(\alpha_x)=l^{\mbox{\tiny so}}$ and so $v^{\mbox{\tiny so}}_2=x-l^{\mbox{\tiny so}}$.
    \item If $x=u^{\mbox{\tiny so}}$, then $x=w_1(\alpha_x)+w_2(\alpha_x)$. Furthermore, $\alpha_x =\alpha^{\mbox{\tiny so}}_2$ and $t_1<\bar{\lambda}(\alpha_x)=t_2$. Then from Theorem \ref{teo:characterization} we get $v^{\mbox{\tiny so}}_1=w_1(\alpha_x)=l^{\mbox{\tiny so}}$ and so $v^{\mbox{\tiny so}}_2=x-l^{\mbox{\tiny so}}=w_2(\alpha_x)$.
    \item If $x>u^{\mbox{\tiny so}}$, then $\alpha_x >\alpha^{\mbox{\tiny so}}_2$ and $\bar{\lambda}(\alpha_x)>t_2$. Then from Theorem \ref{teo:characterization} we get $v^{\mbox{\tiny so}}_1=w_1(\alpha_x)=\mu x$ and $v^{\mbox{\tiny so}}_2=w_2(\alpha_x)=(1-\mu)x$.
\end{itemize}

In summary, the social optimum flow assignment $v^{\mbox{\tiny so}}_1$ and the equilibrium flow assignment $v^{\mbox{\tiny w}}_1$ (calculated in \cite{cominetti2001common}) for this example are:
\begin{align} \label{eq:asignejem}
v_1^{\mbox{\tiny so}}= \begin{cases}x & \text { if } x \leq l^{\mbox{\tiny so}}, \\ l^{\mbox{\tiny so}} & \text { if } l^{\mbox{\tiny so}}<x<u^{\mbox{\tiny so}},\quad  \\ \mu x & \text { if } x \geq u^{\mbox{\tiny so}} ;\end{cases}  &  v_1^{\mbox{\tiny w}}= \begin{cases}x & \text { if } x \leq l^{\mbox{\tiny w}}, \\ l^{\mbox{\tiny w}} & \text { if } l^{\mbox{\tiny w}}<x<u^{\mbox{\tiny w}}, \quad\\ \mu x & \text { if } x \geq u^{\mbox{\tiny w}} ;\end{cases}
\end{align}
where 
$$
 l^{\mbox{\tiny w}}:=\sum_{i=1}^2\left\{w_i(\alpha^{\mbox{\tiny w}}_2): t_i<\widehat{T}(w(\alpha^{\mbox{\tiny w}}_2))\right\} \mbox{ and }
 u^{\mbox{\tiny w}}:=\sum_{i=1}^2\left\{w_i(\alpha^{\mbox{\tiny w}}_2): t_i \leq \widehat{T}(w(\alpha^{\mbox{\tiny w}}_2))\right\},
$$
and $\alpha^{\mbox{\tiny w}}_2$ is the solution of $\widehat{T}(w(\alpha))=t_2$. As we can see, the social optimal assignment takes the same values as the equilibrium assignment, but the difference lies in the threshold values $l$ and $u$. Let us show that the solution of $\bar{\lambda}(\alpha)=t_2$ is less than the solution of $\widehat{T}(w(\alpha))=t_2$ corresponding to the equilibrium problem, and therefore $l^{\mbox{\tiny so}}<l^{\mbox{\tiny w}}$ and $u^{\mbox{\tiny so}}<u^{\mbox{\tiny w}}$.

Since $\alpha^{\mbox{\tiny so}}_2$ is the solution of $\bar{\lambda}(\alpha)=t_2$, it follows that
\eqref{eqn:conditionexample} holds, so $\alpha^{\mbox{\tiny so}}_2$ is the solution of
$$1-\frac{1}{\mu_1(t_2-t_1)}=\left(\frac{\alpha}{1+\alpha}\right)^K \left( 1 + \frac{K}{1+\alpha} \right). $$ 
The right-hand side of the above equation is an increasing continuous function of $\alpha$. Similarly, since $\alpha^{\mbox{\tiny w}}_2$ is the solution of $\widehat{T}(w(\alpha))=t_2$, we get that $\alpha^{\mbox{\tiny w}}_2$ is the solution of 
$$1-\frac{1}{\mu_1(t_2-t_1)}=\left(\frac{\alpha}{1+\alpha}\right)^K. $$
In that case, the right hand is also an increasing continuous function of $\alpha$. Furthermore
$$\left(\frac{\alpha}{1+\alpha}\right)^K < \left(\frac{\alpha}{1+\alpha}\right)^K \left( 1 + \frac{K}{1+\alpha} \right)$$
for all $\alpha \in (0, \infty)$. Therefore, $\alpha^{\mbox{\tiny so}}_2<\alpha^{\mbox{\tiny w}}_2$. Then, since $w_i$' s are increasing functions, we get $l^{\mbox{\tiny so}}<l^{\mbox{\tiny w}}$ and $u^{\mbox{\tiny so}}<u^{\mbox{\tiny w}}$. 

To compute the social cost, we will use the arc flow formulation given by \eqref{eq:social_cost_arcs}. From \eqref{eq:asignejem} we get
\begin{itemize}
    \item If $x \le l^{\mbox{\tiny so}}$,
    $$SC=t_1x + t_20+\max\left( \frac{x}{f_1(x)}, \frac{0}{f_2(0)} \right)=t_1x+\frac{\rho_1(x)}{1-\rho_1(x)}$$
    \item If $l^{\mbox{\tiny so}}<x<u^{\mbox{\tiny so}}$,
    $$SC=t_1l^{\mbox{\tiny so}} + t_2(x-l^{\mbox{\tiny so}})+\max\left( \frac{l^{\mbox{\tiny so}}}{f_1(l^{\mbox{\tiny so}})}, \frac{x-l^{\mbox{\tiny so}}}{f_2(x-l^{\mbox{\tiny so}})} \right)$$
    From the constraint $l^{\mbox{\tiny so}}<x<u^{\mbox{\tiny so}}$, it follows that
    $$ x-l^{\mbox{\tiny so}} < u^{\mbox{\tiny so}}-l^{\mbox{\tiny so}}=w_2(\alpha^{\mbox{\tiny so}}_2)=\alpha^{\mbox{\tiny so}}_2 f_2(w_2(\alpha^{\mbox{\tiny so}}_2)),$$
    where the last equality follows from \eqref{eqn:defw}. Furthermore $f_2(\cdot)$ is a decreasing function, so
    $$ \frac{x-l^{\mbox{\tiny so}}}{f_2(x-l^{\mbox{\tiny so}})}< \frac{u^{\mbox{\tiny so}}-l^{\mbox{\tiny so}}}{f_2(u^{\mbox{\tiny so}}-l^{\mbox{\tiny so}})}=\frac{\alpha^{\mbox{\tiny so}}_2 f_2(w_2(\alpha^{\mbox{\tiny so}}_2))}{f_2(w_2(\alpha^{\mbox{\tiny so}}_2))}=\alpha^{\mbox{\tiny so}}_2.$$
    On the other hand, 
    $$ \frac{l^{\mbox{\tiny so}}}{f_1(l^{\mbox{\tiny so}})}=\frac{w_1(\alpha^{\mbox{\tiny so}}_2)}{f_1(w_1(\alpha^{\mbox{\tiny so}}_2))}=\frac{\alpha^{\mbox{\tiny so}}_2 f_1(w_1(\alpha^{\mbox{\tiny so}}_2))}{f_1(w_1(\alpha^{\mbox{\tiny so}}_2))}=\alpha^{\mbox{\tiny so}}_2,$$
    where the last equation follows again from \eqref{eqn:defw}. Therefore
    $$\max\left( \frac{l^{\mbox{\tiny so}}}{f_1(l^{\mbox{\tiny so}})}, \frac{x-l^{\mbox{\tiny so}}}{f_2(x-l^{\mbox{\tiny so}})} \right)=\alpha^{\mbox{\tiny so}}_2,$$
    and we can rewrite the social optimum as
    $$ SC=l^{\mbox{\tiny so}}t_1 + (x-l^{\mbox{\tiny so}})t_2 + \alpha^{\mbox{\tiny so}}_2=xt_2-l^{\mbox{\tiny so}}(t_2-t_1) + \alpha^{\mbox{\tiny so}}_2 =xt_2-\frac{w_1(\alpha^{\mbox{\tiny so}}_2)}{w'_1(\alpha^{\mbox{\tiny so}}_2)}+ \alpha^{\mbox{\tiny so}}_2.$$
    Finally, using \eqref{eqn:defw} and \eqref{eq:w_derivative} we get
    $$ SC=t_2x+\left(\alpha^{\mbox{\tiny so}}_2\right)^2f_1'(w_1(\alpha^{\mbox{\tiny so}}_2)).$$
    \item If $x \ge u^{\mbox{\tiny so}}$,
    $$SC=t_1 \mu x + t_2(1-\mu)x+\max\left( \frac{\mu x}{f_1(\mu x)}, \frac{(1-\mu)x}{f_2((1-\mu)x)} \right)=t_1 \mu x+ t_2 (1-\mu)x + \frac{\rho_1(\mu x)}{1-\rho_1(\mu x)}$$
    where the last equality follows from the fact that $\rho_1(\mu x)=\rho_2((1-\mu)x)$.
\end{itemize}

In summary, setting $t^{\mu}=\mu t_1 + (1- \mu)t_2$, the optimum social cost, and the equilibrium cost are

\begin{align} \label{eq:costejem_osc}
OSC= \begin{cases} t_1x+\frac{\rho_1(x)}{1-\rho_1(x)} & \text { if } x \leq l^{\mbox{\tiny so}}, \\ t_2x+\left(\alpha^{\mbox{\tiny so}}_2\right)^2f_1'(w_1(\alpha^{\mbox{\tiny so}}_2)) & \text{ if } l^{\mbox{\tiny so}}<x<u^{\mbox{\tiny so}}, \quad \\ t^\mu x+\frac{\rho_1(\mu x)}{1-\rho_1(\mu x)} & \text{ if } x \geq u^{\mbox{\tiny so}} ,\end{cases} 
\end{align}

\begin{align} \label{eq:costejem_wsc}
 WSC = \begin{cases} t_1x+\frac{\rho_1(x)}{1-\rho_1(x)} & \text{ if } x \leq l^{\mbox{\tiny w}}, \\ t_2x & \text{ if } l^{\mbox{\tiny w}}<x<u^{\mbox{\tiny w}}, \\ t^\mu x+\frac{\rho_1(\mu x)}{1-\rho_1(\mu x)} & \text{ if } x \geq u^{\mbox{\tiny w}} .\end{cases}
\end{align}

With this development, we obtained the formulation that allows us to calculate the social optimum assignment and the optimal social cost, equivalent to the equilibrium case reported in the literature. To illustrate the difference between the social optimum assignment and the equilibrium assignment we set the travel times as $t_1=1/4$ and $t_2=1/2$ (in hours), the capacity as $K=20$ passengers per bus, and the arrival rates as $\mu_1 = 16$ and $\mu_2 = 10$ buses per hour. With these data, we were able to explicitly calculate the values for $l^{\mbox{\tiny so}}$, $u^{\mbox{\tiny so}}$, $ l^{\mbox{\tiny w}}$ and $ u^{\mbox{\tiny w}}$, obtaining:

\begin{equation}
    l^{\mbox{\tiny so}} = 202.77, ~~ u^{\mbox{\tiny so}} =329.51, ~~ l^{\mbox{\tiny w}} = 276.09, ~~ u^{\mbox{\tiny w}} =448.65.
\end{equation}
Once we have these values, it is possible to calculate the equilibrium and social optimum assignment according to \eqref{eq:asignejem} and the corresponding social cost for each case. Figure \ref{fig:flows_toy_example} shows the evolution of flows for the equilibrium and the social optimum assignment, while Figure \ref{fig:social_cost_and_poa_toy_example} shows the social costs and the corresponding price of anarchy for this example. In Figure \ref{fig:poa_toy_example}  it can be seen that the price of anarchy is greater than 1 for $ l^{\mbox{\tiny so}} \leq x \leq u^{\mbox{\tiny w}}$, that is, for demands between 202 and 448 passengers (approximately) the network becomes inefficient in an equilibrium situation. Moreover, the maximum price of anarchy occurs when the demand is $x=l^{\mbox{\tiny w}} = 276$ passengers (approximately).

\begin{figure}[t!]
\centering
\subfigure[Equilibrium assignment.]{\includegraphics[width=0.49\textwidth]{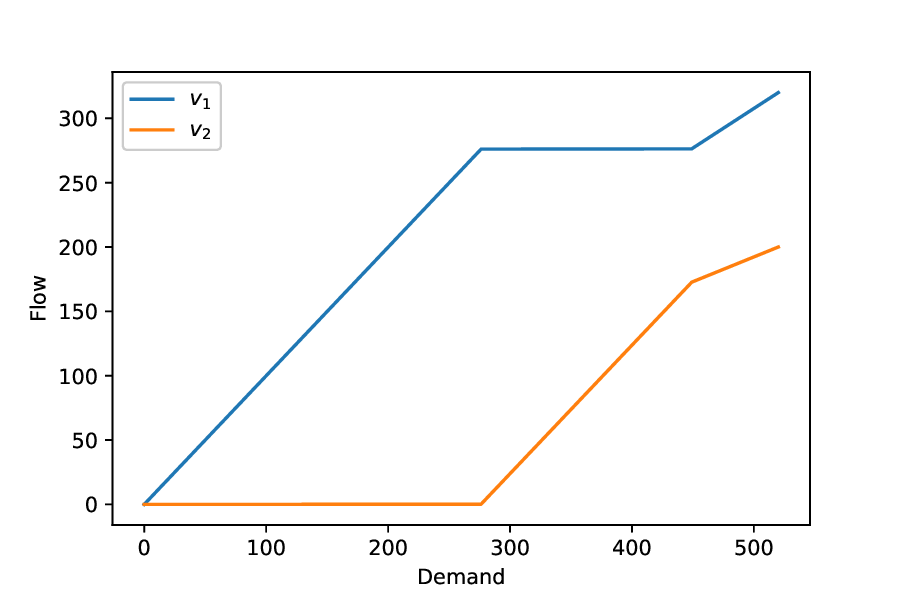}\label{fig:equilibrium_flows_toy_example}}
\subfigure[Social optimum assignment.]{\includegraphics[width=0.49\textwidth]{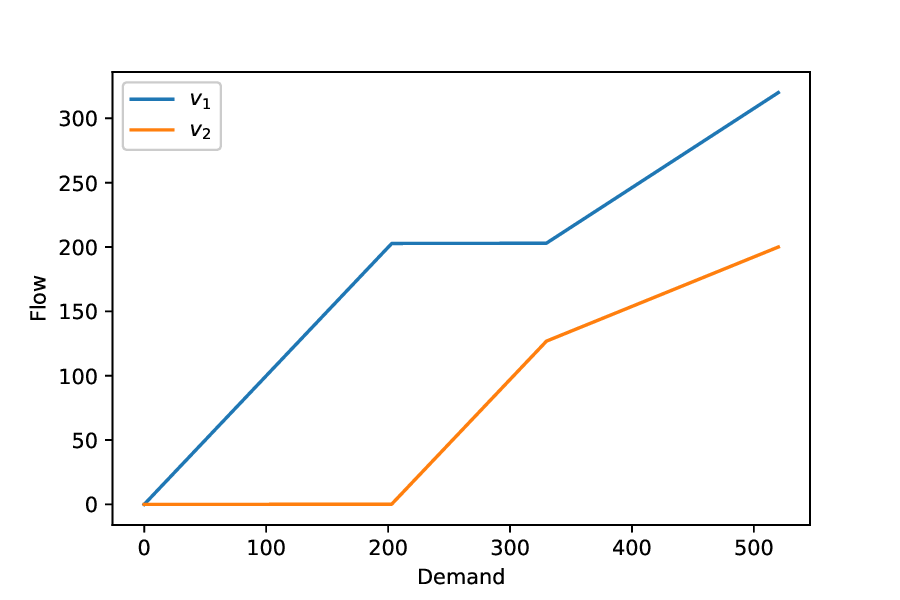}\label{fig:os_flows_toy_example}}
\caption{Evolution of arc flows for the example based on the solution characterisation.} \label{fig:flows_toy_example}
\end{figure}

\begin{figure}[t!]
\centering
\subfigure[Social cost for equilibrium assignments and social optimum.]{\includegraphics[width=0.49\textwidth]{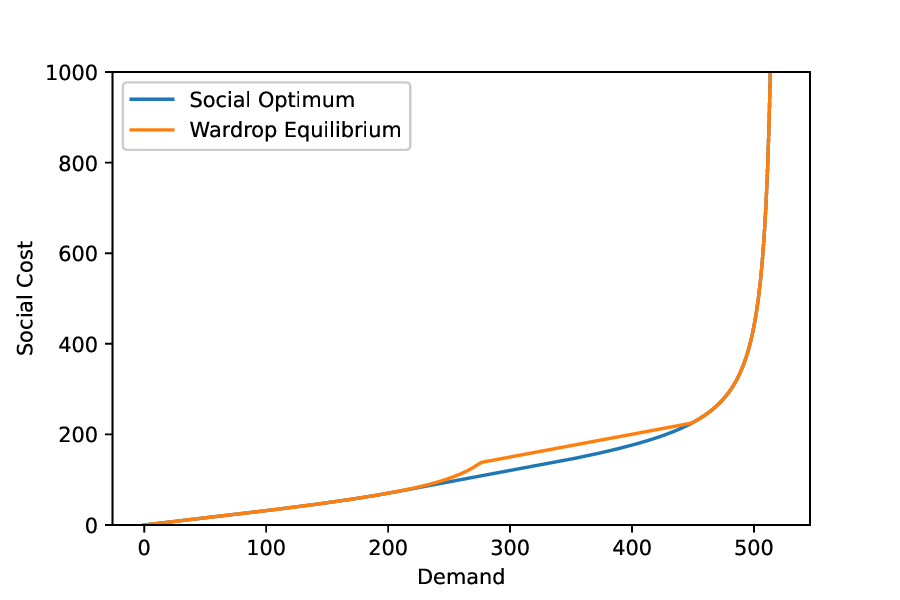}\label{fig:social_cost_toy_example}}
\subfigure[Price of anarchy.]{\includegraphics[width=0.49\textwidth]{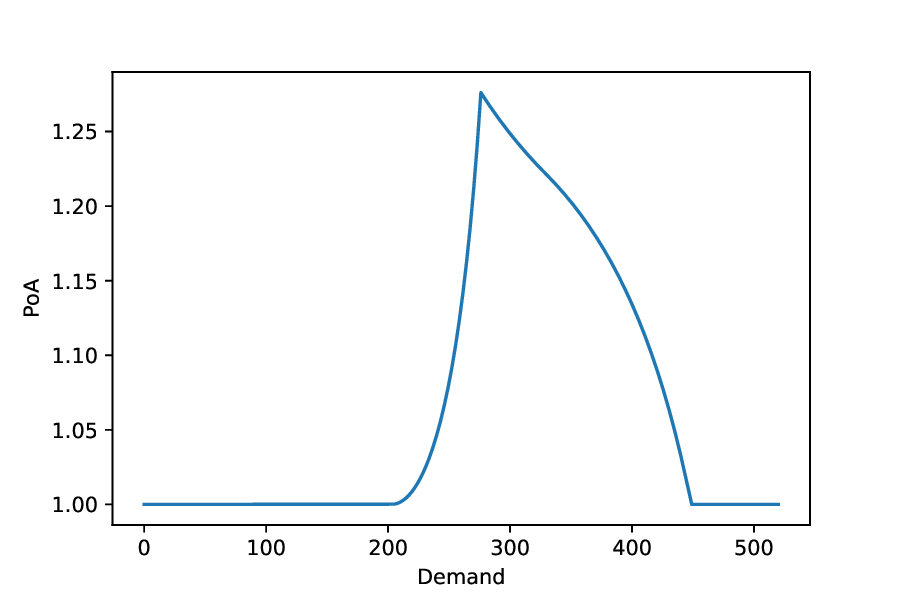}\label{fig:poa_toy_example}}
\caption{Evolution of the social cost and the price of anarchy for the example based on the solution characterisation.} \label{fig:social_cost_and_poa_toy_example}
\end{figure}

\subsection{Example from optimisation problems}
\label{sec:example_optim}

In this example, we will use the effective frequency function proposed in \cite{CCF}. In the case of a network with only two nodes (one origin and one destination), the origin will only have passengers boarding (no alighting), and the destination will only have passengers alighting (no boarding), giving the following frequency function for each arc $i \in A$:

\begin{equation}\label{eq:frequency}
f_i(v_i) = \left\{ \begin{array}{lcl} 
\mu \left[ 1- \left( \frac{v_i}{\mu K} \right)^{\beta} \right] & \text{if} & 0 \leq v_i < \mu K, \\ 
\\ \varepsilon & \text{otherwise,}  
\end{array} \right.
\end{equation}
where $\mu$ is the nominal frequency of the line, $K$ is the bus capacity, and $v_i$ is the flow boarding the line. The exponent $\beta$ is always positive and determines the effective frequency's updating degree. Finally, $\varepsilon$ is a very small value that represents a high waiting time at the stop when the bus does not have the capacity for new passengers to board. As in \cite{CCF}, we will consider $\varepsilon=1/999$. As seen in Appendix \ref{app:f}, the function \eqref{eq:frequency} satisfies all the necessary conditions to formulate the social optimum problem as developed in Section \ref{sec:SO}. For this frequency function, it is challenging to obtain explicitly the function $w$. It could be obtained numerically, or else not to solve the problem from the solution characterisation (as was done in Section \ref{sec:example_charac}) but to solve the corresponding optimisation problem. To obtain the Wardrop equilibrium, we solve the problem proposed in \cite{CCF}, and to obtain the social optimum, we can solve the problem formulated using strategy flows (problem \eqref{eq:social_optimimum_original}) or arc flows (problem \eqref{eq:social_optimum_without_w}).

We consider the network shown in Figure \ref{fig:example_two_nodes} and the same data described at the end of Section \ref{sec:example_charac}. The only additional parameter is the exponent $\beta$, which we consider to be $\beta=0.2$. With this information, in the first instance, we assigned a demand of $x=100$ passengers per hour and performed the equilibrium and social optimum assignments. For the equilibrium assignment, we use a self-regulated method as proposed in \cite{ORLANDO2023100832}. To obtain the social optimum, we solved the original problem \eqref{eq:social_optimimum_original} formulated in terms of strategy flows, and we obtained the same solution as when solving the problem \eqref{eq:social_optimum_without_w}, which only involves the arc flows. To solve the social optimum problems, we employed the SciPy library v1.8.0 (\cite{SciPy_NMeth2020}), which is accessible in Python v3.8.5 (\cite{python}). Specifically, we utilised the \textit{minimize} function within this library, wherein we defined the objective functions and established linear and nonlinear constraints for each case. The trust constraint method was chosen and we configured the maximum number of iterations to be $4000$,  with termination tolerances set at $gtol=10^{-8}$ for the norm of the Lagrangian gradient and $xtol=10^{-8}$ for the change in the independent variable. The results for equilibrium and social optimum are shown in Table \ref{table:results_example_2}.

\begin{table}[b!]
\begin{tabular*}{\hsize}{@{\extracolsep{\fill}}ccc@{}}
\hline
Results &  Equilibrium assignment & Social Optimum \\
\hline
$v_1$  &  $75.94$  & $61.54$  \\ 
$v_2$ & $24.06$  & $38.46$  \\ 
 Social Cost & $50$  & $48.309$ \\  
 \hline
\end{tabular*}
\caption{Arc flow and social cost obtained in the equilibrium assignment and the social optimum.}
\label{table:results_example_2}
\end{table}

For these results, we can calculate the price of anarchy as stated in \eqref{eq:price_of_anarchy}, obtaining:

$$PoA = \frac{WSC}{OSC} =\frac{50}{48.309}=1.035.$$
As the price of anarchy is greater than 1, we can conclude that the system becomes inefficient in an equilibrium situation for the considered demand.

As in Section \ref{sec:example_charac}, it would also be interesting to analyse the system efficiency as demand increases. For this purpose, we make flow assignments for demands between $1$ and $160$ passengers and analyse the progress of the flow assignments, the social cost associated with the equilibrium ($WSC$) and the optimal system cost ($OSC$), and the consequent evolution of the price of anarchy. Figure \ref{fig:flows_second_example} shows the flow assignments, while Figure \ref{fig:cost_and_poa_second_example} exposes the social costs and the price of anarchy. As mentioned above, for the frequency function used in this example, it is not easy to find an expression for the function $w(\alpha)$. However, $w(\alpha)$ and $w'(\alpha)$ can be found numerically, making it possible to calculate:

\begin{equation}
    l^{\mbox{\tiny so}} = 38.59, ~~ u^{\mbox{\tiny so}} =62.72, ~~ l^{\mbox{\tiny w}} = 75.94, ~~ u^{\mbox{\tiny w}} =123.4.
\end{equation}
If we analyse Figure \ref{fig:poa_second_example}, we observe that the network becomes inefficient in the equilibrium situation when $l^{\mbox{\tiny so}} \leq x \leq u^{\mbox{\tiny w}}$. That is to say, if demand varies between 39 and 124 passengers (approximately), the network becomes inefficient when users behave selfishly. In addition, the maximum price of anarchy occurs when $x=l^{\mbox{\tiny w}}$ (approximately 76 passengers).

From what has been exposed in this Section and Section \ref{sec:example_charac}, we can conclude that, although when considering different frequency functions the values of $l^{\mbox{\tiny w}}$, $l^{\mbox{\tiny so}}$, $u^{\mbox{\tiny w}}$ and $u^{\mbox{\tiny so}}$ changed, in both cases the system becomes inefficient in front of the selfish behaviour when $l^{\mbox{\tiny so}} \leq x \leq  u^{\mbox{\tiny w}}$, and reaches the maximum point of inefficiency when $x=l^{\mbox{\tiny w}}$.

\begin{figure}[t!]
\centering
\subfigure[Equilibrium assignment.]{\includegraphics[width=0.49\textwidth]{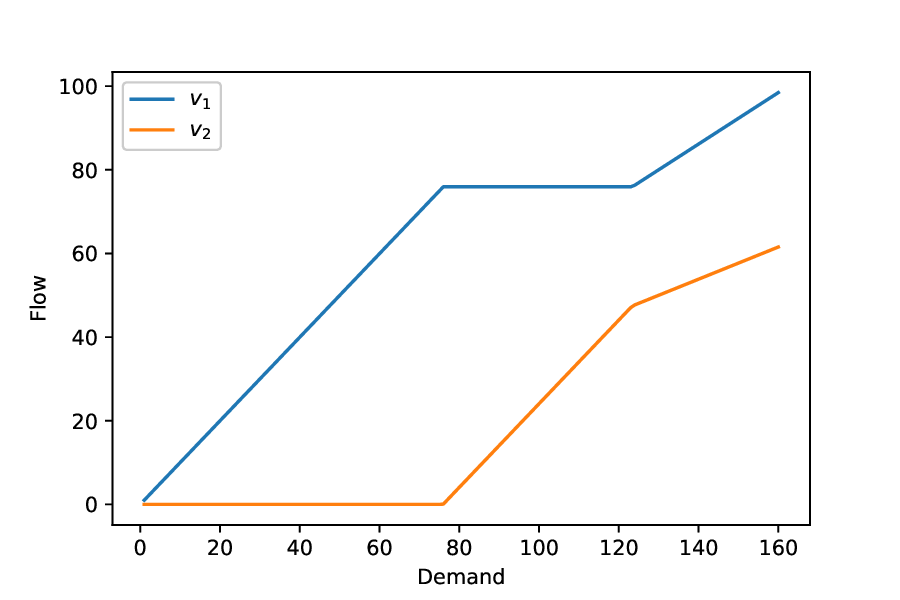}\label{fig:flujos_equilibrio}}
\subfigure[Social optimum assignment.]{\includegraphics[width=0.49\textwidth]{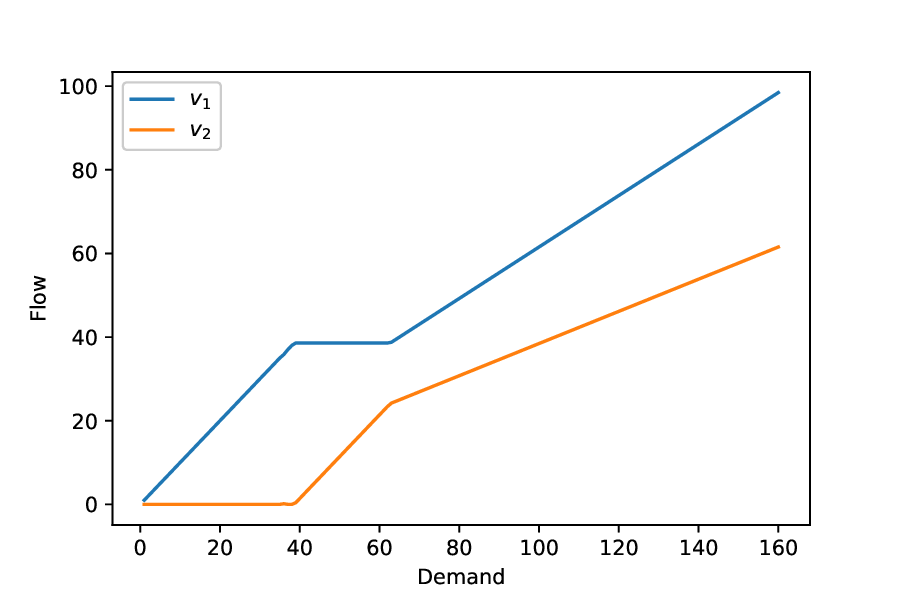}\label{fig:flujos_os}}
\caption{Evolution of arc flows for the example based on the optimisation problem.} \label{fig:flows_second_example}
\end{figure}

\begin{figure}[t!]
\centering
\subfigure[Social cost for equilibrium assignments and social optimum.]{\includegraphics[width=0.49\textwidth]{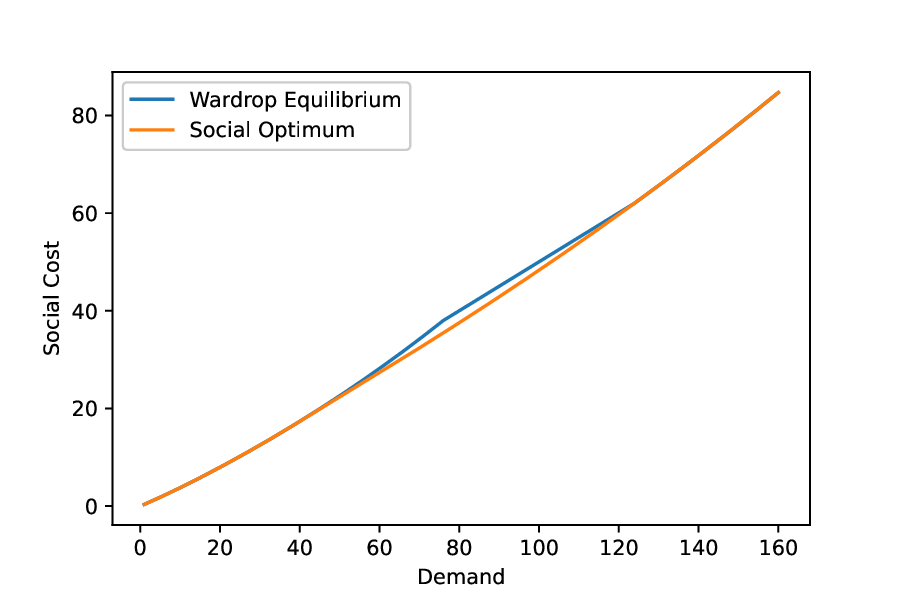}\label{fig:social_costs_example}}
\subfigure[Price of anarchy.]{\includegraphics[width=0.49\textwidth]{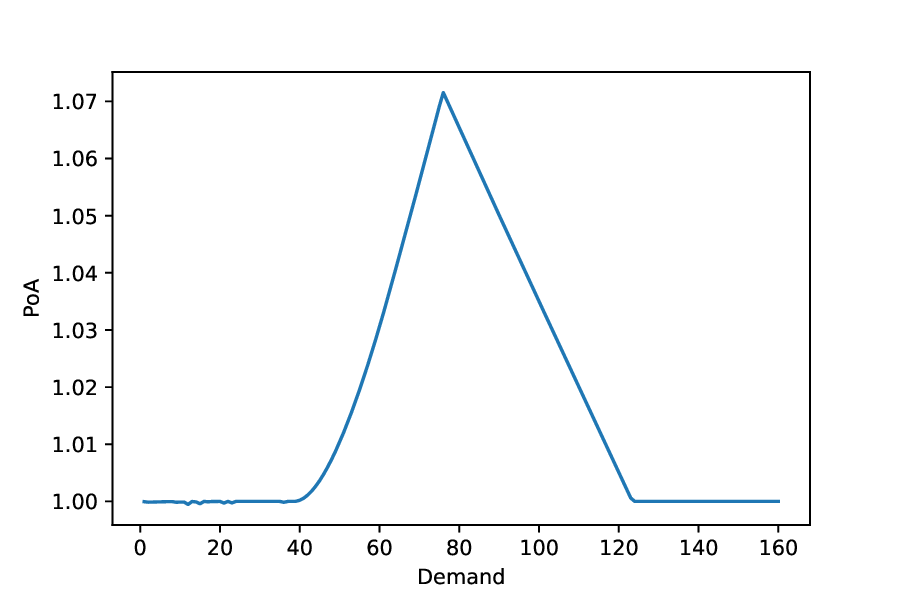}\label{fig:poa_second_example}}
\caption{Evolution of the social cost and the price of anarchy for the example based on the optimisation problem.} \label{fig:cost_and_poa_second_example}
\end{figure}

\section{Conclusions}

Throughout this work, we have presented the social optimum problem applied to the common-lines problem, which had not been studied until now. We presented the approach using strategy flows, and from this approach, we arrived at a formulation that only uses arc flows, which makes it easier to handle and interpret. We have formulated and proved some results that allow us to obtain a characterisation of the solution to the social optimum problem. From it, the analysis of the price of anarchy becomes possible. Although this concept has already been studied in public transportation networks, we did not find an analysis for the common-lines problem since knowledge of the solution that guarantees the optimal operation of the system is required to calculate the price of anarchy.

Finally, to illustrate all the concepts previously developed, two examples were presented to compare the solutions with those found in the literature for the Wardrop Equilibrium and to analyse the evolution of the system efficiency as demand varies. With these examples, we show that we can solve the equilibrium problem and the social optimum problem and analyse the evolution of the social cost in each case, either by solving the corresponding optimisation problems or by working from the characterisation of the solution.

The analysis developed in this work is based on simple networks with one origin node and one destination node. This work is the beginning of a future analysis for general networks, where intermediate nodes are considered that allow passenger boarding, alighting, and line interchanges.

\appendix
\section{Properties of $w_i$} \label{app:w}

Let $w_i:[0,\infty) \to [0,\bar{v}_i)$ be the inverse function of $v_i \mapsto \frac{v_i} {f_i\left(v_i\right)}$ for all $i=1,\dots,n$.

\begin{proposition} \label{prop:w}
$w_i$ is a differentiable function, with $w'_i(\alpha)>0$ for all $\alpha \in (0,\infty)$ (and then strictly increasing). Furthermore, if $f_i$ is $C^1[0,\bar{v}_i)$ then $w'_i$ is continuous. 
\end{proposition}
\begin{proof}
    Since $w_i$ is the inverse function of $v_i \mapsto \frac{v_i} {f_i\left(v_i\right)}$, then $\frac{w_i(\alpha)}{f_i(w_i(\alpha))}=\alpha $ and
    \begin{equation} \label{eqn:defw}
        w_i(\alpha) = \alpha f_i(w_i(\alpha)). 
        \end{equation}
    Differentiating the last expression, we have
    $$ w'_i(\alpha) = f_i(w_i(\alpha)) +\alpha f'_i(w_i(\alpha))w'_i(\alpha).$$
   Since $f'_i(v_i)<0$ for all $v_i \in [0,\bar{v}_i)$ and $\alpha \ge 0$, it follows that
    \begin{equation}\label{eq:w_derivative}
        w'_i(\alpha) = \frac{f_i(w_i(\alpha))}{1-\alpha f'_i(w_i(\alpha))}
    \end{equation}
    and then $w_i$ is differentiable. Furthermore, since $f_i(v_i)>0$ for all $v_i \in [0,\bar{v}_i)$, it follows that $w'_i(\alpha)>0$. Finally, if $f_i(\cdot)$ and $f_{i}'(.)$ are continuous, it is easy to see from \eqref{eq:w_derivative} that $w_i'$ is also continuous. 
\end{proof}

\begin{proposition}\label{prop:concavity_w}
    Let $f_i:\left[0, \bar{v}_i\right) \rightarrow(0, \infty)$ be a family of frequency functions for each $i = 1, \dots,n$ such that $f_i''$ exists. Then, $w_i$ is concave if and only if  
    $$2f_i'(w_i(\alpha)) + \alpha f_i''(w_i(\alpha))w_i'(\alpha)<0.$$ 
\end{proposition}

\begin{proof}
    From the previous proposition, we know that 
    $$ w'_i(\alpha) = f_i(w_i(\alpha)) +\alpha f'_i(w_i(\alpha))w'_i(\alpha).$$
    Deriving on both sides, we obtain
    $$w_i''(\alpha) = f_i'(w_i(\alpha))w_i'(\alpha)+\left[ f_i'(w_i(\alpha)) + \alpha f_i''(w_i(\alpha))w_i'(\alpha)\right] w_i'(\alpha) + \alpha f_i'(w_i(\alpha))w_i''(\alpha)$$
    from where we can get
    $$w_i''(\alpha) = \frac{2f_i'(w_i(\alpha))w_i'(\alpha) + \alpha f_i''(w_i(\alpha))(w_i'(\alpha))^2}{1-\alpha f_i'(w_i(\alpha))}.$$
    Since $f_i'(w_i(\alpha))<0$ and $\alpha \geq 
0$, $1-\alpha f_i'(w_i(\alpha))$ is always positive, to ensure that $w_i$ is concave, it is only necessary that 
    $$2f_i'(w_i(\alpha))w_i'(\alpha) + \alpha f_i''(w_i(\alpha))(w_i'(\alpha))^2<0.$$
From Proposition \ref{prop:w} we know that $w_i'(\alpha)>0$ for all $\alpha \in [0, \infty)$, so the above inequality is equivalent to:
    $$2f_i'(w_i(\alpha)) + \alpha f_i''(w_i(\alpha))w_i'(\alpha)<0.$$

\end{proof}

Note that from \eqref{eq:w_derivative}, $w'_i$ is a frequency function for line $i$. Moreover, since $1-\alpha f'_i(w_i(\alpha))>1$, then $w_i'(\alpha)<f_i(\alpha).$ Let's see that $w'_i$ holds the requested properties for frequencies.

\begin{proposition} \label{prop:propdw}
    Assume that the function $w_i\in C^2(0,\infty)$ is strictly concave. Then $w_i'$ is a decreasing smooth function defined on $(0, \infty)$ with $w'_i(\alpha) \to 0$ when $\alpha \to \infty$. 
\end{proposition}
\begin{proof}
    If the function $w_i$ are strictly concave, then $w''_i(\alpha)<0$ for all $\alpha \in \mathbb{R}$. Then $w'_i$ is decreasing. Furthermore, $\lim\limits_{\alpha \to \infty}w_i'(\alpha)=0$ since $\lim\limits_{\alpha \to \infty}w_i(\alpha)=\bar{v}_i$ and $w_i$ is concave.
\end{proof}

\section{Properties of $\psi_{\alpha}$ function} \label{sec:app}

\begin{proposition} \label{prop:psidecr}
    Assume that the functions $w_i\in C^2(0,\infty)$ are strictly concave for all $i=1,\dots,n$. Then, for each $\lambda \ge t_1$, the function $\psi_{\alpha}(\lambda)$ is decreasing and continuous as a function of $\alpha$.
\end{proposition}
\begin{proof}
    It is easy to see that $\psi_{\alpha}(\lambda)$ is continuous for fixed $\lambda$. Let us now $\lambda \ge t_1$. In particular, $\lambda \ge t_k$ for some $k$. If for each $i$, the functions $w_i$ are strictly concave, then $w''_i(\alpha)<0$ for all $\alpha \in \mathbb{R}$. Therefore, $$\frac{d}{d\alpha} \psi_{\alpha}(\lambda)=\sum_{i=1}^{k-1}(\lambda-t_i)w''_i(\alpha) <0$$  
    and the function $\psi_{\alpha}(\lambda)$ is strictly decreasing as a function of $\alpha$.
\end{proof}

We are interested in the values of $\alpha \in (0, \infty)$ that make equation \eqref{eq:lambdaalpha} have a solution.

\begin{proposition} \label{prop:chain_alpha}
    Assume that the functions $w_i\in C^2(0,\infty)$ are strictly concave for all $i=1,\dots,n$. Let $\lambda^*$, if there exists $\alpha^* \in (0,\infty)$ such that $\psi_{\alpha^*}(\lambda^*)=1$, then equation \eqref{eq:lambdaalpha} has a solution for all $\lambda>\lambda^*$. On the other hand, if there is no $\alpha^* \in (0,\infty)$ such that $\psi_{\alpha^*}(\lambda^*)=1$, then equation \eqref{eq:lambdaalpha} does not have a solution for all $\lambda<\lambda^*$.
\end{proposition}
\begin{proof}
     Suppose first that there exists $\alpha^* \in (0,\infty)$ such that $\psi_{\alpha^*}(\lambda^*)=1$ for some $\lambda^*$. Then
     $$\psi_{\alpha^*}(\lambda)=\sum_{i=1}^{n}(\lambda-t_i)_+w'_i(\alpha^*) > \sum_{i=1}^{n}(\lambda^*-t_i)w'_i(\alpha^*) = 1 \quad \forall~ \lambda>\lambda^*,$$
     and $\psi_{\alpha}(\lambda)$ is continuous and strictly decreasing in $\alpha$ by Proposition \ref{prop:psidecr}, then there exist $\alpha>\alpha^*$ such that $\psi_{\alpha}(\lambda)=1$ for all $ \lambda>\lambda^*$. On the other hand, suppose there is no $\alpha^*$ such that $\psi_{\alpha^*}(\lambda^*)=1$. Since the function $\psi_{\alpha}(\lambda^*)$ is decreasing as a function of $\alpha$ by Proposition \ref{prop:psidecr}, then $\psi_{\alpha}(\lambda^*)<1$ for all $\alpha \in (0,\infty)$.Then
     $$ 1 > \sum_{i=1}^{n}(\lambda^*-t_i)_+w'_i(\alpha)> \sum_{i=1}^{n}(\lambda-t_i)w'_i(\alpha)=\psi_{\alpha}(\lambda) \quad \forall~ \lambda<\lambda^*.$$
     Therefore, there is no $\alpha$ such that $\psi_{\alpha}(\lambda)=1$ for all $ \lambda<\lambda^*$.
\end{proof}

\begin{lemma} \label{lem:alphak}
Assume that the functions $w_i\in C^2(0,\infty)$ are strictly concave for all $i=1,\dots,n$. Then, for each $\lambda \ge \frac{1+t_1f_1(0)}{f_1(0)}$, there exists an $\alpha \in (0, \infty)$ such that $\psi_\alpha(\lambda)=1$. Furthermore, if $\max_i[(\lambda-t_i)_+ f_i(0)] < (n-1)^{-1}$, then there is no $\alpha \in (0, \infty)$ such that $\psi_\alpha(\lambda)=1$.
\end{lemma}
\begin{proof}
Fix first $\lambda \ge \frac{1+t_1f_1(0)}{f_1(0)}$, then $$(\lambda-t_1)f_1(0)\geq 1.$$ 
In particular, $\lambda \ge t_k$ for some $k$. Then, $$\sum_{i=1}^{k-1}(\lambda-t_i)f_i(0)\geq (\lambda-t_1)f_1(0) \ge 1.$$
Analysing $\lim\limits_{\alpha \to 0^+}w_i'(\alpha)$, we have: 
$$\lim\limits_{\alpha \to 0^+}w_i'(\alpha)= \lim\limits_{\alpha \to 0^+} \frac{f_i(\overset{\to 0}{\overbrace{w_i(\alpha)}})}{1- \underset{\to 0}{\underbrace{\alpha}} f_i'(w_i(\alpha))}=\frac{f_i(0)}{1}=f_i(0).$$
Therefore, it follows from the above inequality and the continuity of $w'$ that
$$\psi_{0}(\lambda)= \lim\limits_{\alpha \to 0^+}\sum_{i=1}^{k-1}(\lambda-t_i)w'_i(\alpha)  = \sum_{i=1}^{k-1}(\lambda-t_i)w'_i(0) = \sum_{i=1}^{k-1}(\lambda-t_i)f_i(0) \geq 1.$$
By hypothesis, the function $\psi_{\alpha}(\lambda)$ is continuous and strictly decreasing as a function of $\alpha$ by Proposition \ref{prop:psidecr}. Furthermore, $\lim\limits_{\alpha \to \infty}w_i'(\alpha)=0$ by Proposition \ref{prop:propdw}. Therefore, 
$$\lim\limits_{\alpha \to \infty} \psi_{\alpha}(\lambda)=\lim\limits_{\alpha \to \infty}\sum_{i=1}^{k-1}(\lambda-t_i)w'_i(\alpha)  =0.$$ 
Hence, since $\lim\limits_{\alpha \to 0^+} \psi_{\alpha}(\lambda) \geq 1$, $\lim\limits_{\alpha \to \infty} \psi_{\alpha}(\lambda) =0$ and $\psi_{\alpha}(\lambda)$ is strictly decreasing as a function of $\alpha$, there exist $\alpha$ such that is the solution of the equation $$\psi_{\alpha}(\lambda)=\sum_{i=1}^{k-1}(\lambda-t_i)w'_i(\alpha) = 1.$$

Assume now $\max_i[(\lambda-t_i)_+ f_i(0)] < (n-1)^{-1}$. Then
\begin{equation*}
    \begin{split}
    1 > (n-1)\max_i[(\lambda-t_i)_+ f_i(0)]  &> \sum_{i=1}^{n-1}(\lambda-t_i)_+f_i(0) \\
    &>\sum_{i=1}^{n-1}(\lambda-t_i)_+w'_i(\alpha) = \psi_{\alpha}(\lambda)
    \end{split}
\end{equation*}
for each $\alpha \in (0, \infty)$. Then there is no $\alpha$ such that $\psi_{\alpha}(\lambda)=1$.
\end{proof}

The values $t_k$ are relevant to the problem. Therefore, we must study the values $\alpha \in (0, \infty)$ such that $\psi_{\alpha}(t_k)=1$ for each $k$.

\begin{corollary} \label{cor:alpha_k}
   Assume that the functions $w_i\in C^2(0,\infty)$ are strictly concave for all $i=1,\dots,n$, and there exists an index $2\le j \le n$ such that $t_j \ge \frac{1+t_1f_1(0)}{f_1(0)}$. Then there exist $\alpha_k \in (0, \infty)$ such that $\psi_{\alpha_k}(t_k)=1$ for all $k=j,\dots,n$.
\end{corollary}
\begin{proof} 
Since $t_j<t_{j+1}<\dots<t_n$ and $t_j \ge \frac{1+t_1f_1(0)}{f_1(0)}$, the proof follows immediately from Lemma \ref{lem:alphak} and Proposition \ref{prop:chain_alpha}.
\end{proof}

\begin{remark} \label{obs:t_k}
    If $(t_2-t_1)f_1(0)\geq 1,$ we get that 
    $$ t_2 \ge \frac{1+t_1f_1(0)}{f_1(0)}.$$
    If $w_i$ is strictly concave for all $i=1,\dots,n$, the Corollary \ref{cor:alpha_k} hypothesis holds. Therefore, for all $k=2,\dots,n$, there exists $\alpha_k \in (0, \infty)$ such that $\psi_{\alpha_k}(t_k)=1$.
\end{remark}

\section{Properties of $f_i$}
\label{app:f}

\begin{proposition}
    The effective frequency function $f_i:[0,\bar{v}_i) \to (0, \infty)$ is strictly decreasing and tends to 0 when $v_i \to \bar{v}_i$.
\end{proposition}

\begin{proof}
    For the formulation presented in equation \eqref{eq:frequency}, the saturation flow of the line $i$ is $\bar{v}_i = \mu c$. If $v_i \in [0,\bar{v}_i)$, we have $0 \leq v_i <\mu c$, so:

    $$f_i(v_i)=\mu \left[ 1- \left( \frac{v_i}{\mu c} \right)^{\beta} \right] = \mu - \frac{\mu}{(\mu c)^{\beta}}v_i^{\beta}.$$
    Therefore, $f_i'$ will be given by:

    $$f_i'(v_i) = - \frac{\mu}{(\mu c)^{\beta}} \beta v_i^{\beta-1},$$
    and is $f_i'(v_i)<0$ since $\mu, c, \beta, v_i > 0$. In particular, $f_i'(v_i)=0$ iff $v_i=0$.
    To see that the frequency of an arc tends to 0 when the flow on it tends to its maximum capacity, it is only necessary to note that $v_i/\mu c \to 1$ when $v_i \to \mu c$.
\end{proof}

\begin{proposition}
    The effective frequency function given in \eqref{eq:frequency} is concave iff $\beta >1 $ and convex iff $0<\beta <1$. 
\end{proposition}

\begin{proof}
    From equation \eqref{eq:frequency}, we have, for $v_i \in [0,\mu c)$:
    $$f_i'(v_i) = - \frac{\mu}{(\mu c)^{\beta}} \beta v_i^{\beta-1},$$
from which we can obtain:
$$f_i''(v_i)= - \frac{\mu}{(\mu c)^{\beta}} \beta (\beta -1) v_i^{\beta-1}.$$
Knowing that $\mu, c, v_i > 0$, it is not difficult to see that $f_i''(v_i)<0$ iff $\beta>1$ and $f_i''(v_i)>0$ iff $0<\beta<1$.
    
\end{proof}

To show that $w_i'$ is continuous in Proposition \ref{prop:w}, we need $f_i$ to be $C^1$. We will show that next.

\begin{proposition}
The effective frequency function given in \eqref{eq:frequency} is $C^1$ in $(0,\mu c)$.
\end{proposition}

\begin{proof}
    For $v_i \in (0,\mu c)$, we have $f_i(v_i)=\mu \left[ 1- \left( \frac{v_i}{\mu c} \right)^{\beta} \right]$. It is a continuous function since $\beta >0$ and $v_i/\mu c > 0$.

    As previously shown, $f_i'(v_i) = - \frac{\mu}{(\mu c)^{\beta}} \beta v^{\beta-1}$. If $\beta>1$ (that is, $\beta-1>0$), it is evident that $f_i'(v_i)$ is continuous since $v_i > 0$. If $0<\beta<1$, this implies that $\beta-1<0$, but $f_i'(v_i)$ is still continuous since $v_i > 0$. 
\end{proof}

Another important assumption is that $w_i$ is concave for each $i \in A$. We prove in Proposition \ref{prop:concavity_w} that $w_i$ will be concave if and only if the frequency function satisfies $2f_i'(w_i(\alpha))w_i'(\alpha) + \alpha f_i''(w_i(\alpha))(w_i'(\alpha))^2<0$. To show that the effective frequency \eqref{eq:frequency} satisfies this condition, let us first note that, by definition, $w_i(\alpha) = v_i$ and $\alpha=\frac{v_i}{f_i(v_i)}$. Considering that, we can rewrite $w_i'$ as:

$$w'_i(\alpha) = \frac{f_i(w_i(\alpha))}{1-\alpha f'_i(w_i(\alpha))} = \frac{f_i(v_i)}{1- \frac{v_i}{f_i(v_i)}f_i'(v_i)} = \frac{f_i^2(v_i)}{f_i(v_i)-v_if_i'(v_i)}.$$
Let us see if $f_i$ satisfies the condition established in Proposition \ref{prop:concavity_w}:

\begin{equation*}
    \begin{split}
        2f_i'(w_i(\alpha))w_i'(\alpha) + \alpha f_i''(w_i(\alpha))(w_i'(\alpha))^2 &= \frac{2f_i'(v_i)f_i^2(v_i)}{f_i(v_i)-v_if_i'(v_i)} + \frac{v_i}{f_i(v_i)} \frac{f_i''(v_i)f_i^4(v_i)}{(f_i(v_i)-v_if_i'(v_i))^2} \\
         &= \frac{2f_i'(v_i)f_i^2(v_i)}{f_i(v_i)-v_if_i'(v_i)} +  \frac{v_i f_i''(v_i)f_i^3(v_i)}{(f_i(v_i)-v_if_i'(v_i))^2} \\
        &= \frac{2f_i'(v_i)f_i^2(v_i)(f_i(v_i)-v_if_i'(v_i)) + v_i f_i''(v_i)f_i^3(v_i)}{(f_i(v_i)-v_if_i'(v_i))^2},
    \end{split} 
\end{equation*}
Therefore, $2f_i'(w_i(\alpha))w_i'(\alpha) + \alpha f_i''(w_i(\alpha))(w_i'(\alpha))^2<0$ is equivalent to proving that:
$$2f_i'(v_i)f_i^3(v_i)-2v_if_i^2(v_i)(f_i'(v_i))^2 + v_if_i''(v_i)f_i^3(v_i)<0$$
which, since $f_i(v_i)>0$, is equivalent to proving that:
\begin{equation}\label{eq:condition_concavity_w}
    2f_i'(v_i)f_i(v_i)-2v_i(f_i'(v_i))^2 + v_if_i''(v_i)f_i(v_i)>0.
\end{equation}

\begin{proposition}
    The frequency function \eqref{eq:frequency} satisfies the condition \eqref{eq:condition_concavity_w}.
\end{proposition}

\begin{proof}
    Proving that the frequency function satisfies \eqref{eq:condition_concavity_w} is equivalent to proving that:
    $$-v_if_i''(v_i)f_i(v_i) - 2f_i'(v_i)f_i(v_i) + 2v_i(f_i'(v_i))^2>0.$$
    This is equivalent to proving that:

    \begin{equation*}
        \begin{split}
        \frac{v_i \beta (\beta-1)v_i^{\beta -2}}{c^{\beta}\mu^{\beta -1}}  \left(\mu - \frac{v_i^{\beta}}{c^{\beta}\mu^{\beta -1}} \right) +2 \left(\mu - \frac{v_i^{\beta}}{c^{\beta}\mu^{\beta -1}}\right) \frac{\beta v_i^{\beta-1}}{c^{\beta}\mu^{\beta -1}}  + 2v_i \left(- \frac{1}{c^{\beta}\mu^{\beta -1}} \beta v_i^{\beta-1} \right)^2>0 \Leftrightarrow \\
        \frac{\beta (\beta-1)v_i^{\beta -1}}{c^{\beta}\mu^{\beta -1}}  \left(\mu - \frac{v_i^{\beta}}{c^{\beta}\mu^{\beta -1}} \right) +2 \left(\mu - \frac{v_i^{\beta}}{c^{\beta}\mu^{\beta -1}}\right) \frac{\beta v_i^{\beta-1}}{c^{\beta}\mu^{\beta -1}} + 2v_i \frac{\beta ^2 v_i^{2 \beta -2}}{c^{2 \beta} \mu^{2 \beta -2}}>0.  \\
        \end{split} 
    \end{equation*}
    Multiplying by $c^{\beta}\mu^{\beta -1} >0$, the above is equivalent to:

    \begin{equation*}
        \begin{split}
            \beta (\beta-1)v_i^{\beta-1}\left(\mu - \frac{v_i^{\beta}}{c^{\beta}\mu^{\beta -1}}\right) + 2 \left(\mu - \frac{v_i^{\beta}}{c^{\beta}\mu^{\beta -1}}\right) \beta v_i^{\beta -1} +2 \frac{\beta^2 v_i^{2\beta -1}}{c^{\beta} \mu^{\beta-1}} >0 \Leftrightarrow \\
            \mu \beta (\beta-1)v_i^{\beta-1} - \frac{\beta (\beta-1)v_i^{2 \beta-1}}{c^{\beta}\mu^{\beta -1}} + 2 \mu \beta v_i^{\beta-1} - \frac{2 \beta v_i^{2 \beta -1}}{c^{\beta}\mu^{\beta -1}} + \frac{2 \beta^2 v_i^{ 2\beta-1}}{c^{\beta}\mu^{\beta -1}} >0 \Leftrightarrow \\
            v_i^{\beta-1} \left[ \mu \beta (\beta-1)+2 \mu \beta \right] + v_i^{2 \beta-1} \left[ \frac{- \beta (\beta-1)}{c^{\beta}\mu^{\beta -1}} - \frac{2 \beta}{c^{\beta}\mu^{\beta -1}} + \frac{2 \beta^2}{c^{\beta}\mu^{\beta -1}}\right] >0 \Leftrightarrow \\
            v_i^{\beta-1} \mu \beta (\beta +1) + v_i^{2 \beta-1} \frac{\beta(\beta-1)}{c^{\beta}\mu^{\beta -1}} >0. 
        \end{split}
    \end{equation*}
    Dividing by $v_i^{\beta-1} \mu \beta >0$, the above is equivalent to:

    \begin{equation*}
        \begin{split}
            (\beta+1) + v_i^{\beta} \frac{(\beta-1)}{c^{\beta}\mu^{\beta}} >0  \Leftrightarrow \\
            (\beta+1) + (\beta-1) \left( \frac{v_i}{\mu c} \right)^{\beta}>0 \Leftrightarrow \\
            \beta \left[ 1 + \left( \frac{v_i}{\mu c} \right)^{\beta}\right] + \left[ 1 - \left( \frac{v_i}{\mu c} \right)^{\beta}\right] >0
        \end{split}
    \end{equation*}
    but this is true since $v_i, \mu, c, \beta >0$ and we are considering $v_i < \mu c$, so $v_i/\mu c<1$ and therefore $\left[ 1 - \left( \frac{v_i}{\mu c} \right)^{\beta}\right] >0$.

\end{proof}

\bibliographystyle{plain}
\bibliography{biblio}

\begin{thebibliography}{10}

\bibitem{Agard2006}
B.~Agard, C.~Morency, and M.~Trépanier.
\newblock Mining public transport user behaviour from smart card data.
\newblock {\em IFAC Proceedings Volumes}, 39(3):399--404, 2006.

\bibitem{CCF}
M.~Cepeda, R.~Cominetti, and M.~Florian.
\newblock A frequency-based assignment model for congested transit networks with strict capacity constraints: characterization and computation of equilibria.
\newblock {\em Transportation Research Part B: Methodological}, 40(6):437--459, 2006.

\bibitem{Chriqui1975}
Claude Chriqui and Pierre~N. Robillard.
\newblock Common bus lines.
\newblock {\em Transportation Science}, 9:115--121, 1975.

\bibitem{cominetti2001common}
Roberto Cominetti and Jos{\'e} Correa.
\newblock Common-lines and passenger assignment in congested transit networks.
\newblock {\em Transportation science}, 35(3):250--267, 2001.

\bibitem{Connors}
RD~Connors.
\newblock The price of anarchy in urban traffic networks.
\newblock {\em Mathematics Today}, 52(5):234--237, October 2016.
\newblock {\copyright} 2016. This is an author produced version of a paper published in Mathematics Today. Uploaded with permission from the publisher.

\bibitem{correawardrop}
José~R. Correa and Nicolás~E. Stier-Moses.
\newblock {\em Wardrop Equilibria}.
\newblock John Wiley \& Sons, Ltd, 2011.

\bibitem{KNIGHT2013122}
Vincent~A. Knight and Paul~R. Harper.
\newblock Selfish routing in public services.
\newblock {\em European Journal of Operational Research}, 230(1):122--132, 2013.

\bibitem{Nnene2023}
Obiora Nnene, Johan Joubert, and Mark Zuidgeest.
\newblock A simulation-based optimization approach for designing transit networks.
\newblock {\em Public Transport}, 15, 01 2023.

\bibitem{ORLANDO2023100832}
Victoria~M. Orlando, Enrique~G. Baquela, Neila Bhouri, and Pablo~A. Lotito.
\newblock Public transport demand estimation by frequency adjustments.
\newblock {\em Transportation Research Interdisciplinary Perspectives}, 19:100832, 2023.

\bibitem{Papadimitriou2001}
Christos Papadimitriou.
\newblock Algorithms, games, and the internet.
\newblock {\em Conference Proceedings of the Annual ACM Symposium on Theory of Computing}, 2076, 05 2001.

\bibitem{Rough2002}
Tim Roughgarden and Éva Tardos.
\newblock How bad is selfish routing?
\newblock {\em Journal of the ACM}, 49(2):236--259, 2002.

\bibitem{Shakeel2019}
N.~Shakeel, F.~Baig, and M.~Saddiq.
\newblock Modeling commuter’s sociodemographic characteristics to predict public transport usage frequency by applying supervised machine learning method.
\newblock {\em Transport Technic and Technology}, 15(2):1--7, 2019.

\bibitem{Skhosana2021}
Menzi Skhosana and Absalom Ezugwu.
\newblock A real-time machine learning-based public transport bus-passenger information system.
\newblock {\em International Journal of Sustainable Development and Planning}, 16:1221--1238, 11 2021.

\bibitem{Spengler2023}
Lukas Spengler, Eva Gößwein, Ingmar Kranefeld, Magnus Liebherr, Frédéric Kracht, Dieter Schramm, and Marc Gennat.
\newblock From modeling to optimizing sustainable public transport: A new methodological approach.
\newblock {\em Sustainability}, 15:8171, 05 2023.

\bibitem{spiess1984contributions}
H.~Spiess.
\newblock {\em Contributions {\`a} la th{\'e}orie et aux outils de planification des r{\'e}seaux de transport urbain}.
\newblock Publication (Universit{\'e} de Montr{\'e}al. Centre de recherche sur les transports). Universit{\'e} de Montr{\'e}al, D{\'e}partement d'informatique et de recherche op{\'e}rationnelle - Facult{\'e} des arts et des sciences, 1984.

\bibitem{SPIESS198983}
Heinz Spiess and Michael Florian.
\newblock Optimal strategies: A new assignment model for transit networks.
\newblock {\em Transportation Research Part B: Methodological}, 23(2):83--102, 1989.

\bibitem{python}
Guido Van~Rossum and Fred~L. Drake.
\newblock {\em Python 3 Reference Manual}.
\newblock CreateSpace, Scotts Valley, CA, 2009.

\bibitem{SciPy_NMeth2020}
Pauli Virtanen, Ralf Gommers, Travis~E. Oliphant, Matt Haberland, Tyler Reddy, David Cournapeau, Evgeni Burovski, Pearu Peterson, Warren Weckesser, Jonathan Bright, St{\'e}fan~J. {van der Walt}, Matthew Brett, Joshua Wilson, K.~Jarrod Millman, Nikolay Mayorov, Andrew R.~J. Nelson, Eric Jones, Robert Kern, Eric Larson, C~J Carey, {\.I}lhan Polat, Yu~Feng, Eric~W. Moore, Jake {VanderPlas}, Denis Laxalde, Josef Perktold, Robert Cimrman, Ian Henriksen, E.~A. Quintero, Charles~R. Harris, Anne~M. Archibald, Ant{\^o}nio~H. Ribeiro, Fabian Pedregosa, Paul {van Mulbregt}, and {SciPy 1.0 Contributors}.
\newblock {{SciPy} 1.0: Fundamental Algorithms for Scientific Computing in Python}.
\newblock {\em Nature Methods}, 17:261--272, 2020.

\bibitem{vos2020modeling}
J.~Vos, E.~Waygood, and L.~Letarte.
\newblock Modeling the desire for using public transport.
\newblock {\em Travel Behaviour and Society}, 19:90--98, 2020.

\bibitem{Youn2008}
Hyejin Youn, Michael Gastner, and Hawoong Jeong.
\newblock Price of anarchy in transportation networks: Efficiency and optimality control.
\newblock {\em Physical review letters}, 101:128701, 10 2008.

\bibitem{Zhang2016}
Jing Zhang, Sepideh Pourazarm, Christos~G. Cassandras, and Ioannis~Ch. Paschalidis.
\newblock The price of anarchy in transportation networks by estimating user cost functions from actual traffic data.
\newblock {\em 2016 IEEE 55th Conference on Decision and Control (CDC)}, pages 789--794, 2016.

\bibitem{Zang2018}
Jing Zhang, Sepideh Pourazarm, Christos~G. Cassandras, and Ioannis~Ch. Paschalidis.
\newblock The price of anarchy in transportation networks: Data-driven evaluation and reduction strategies.
\newblock {\em Proceedings of the IEEE}, 106(4):538--553, 2018.

\bibitem{sipus2022defining}
D.~Šipuš, B.~Abramović, and M.~Jakovčić.
\newblock Defining equity criteria for determining fare zones in integrated passenger transport.
\newblock {\em Journal of Advanced Transportation}, pages 1--9, 2022.

\end{thebibliography}

\end{document}